\documentclass[]{article}

\title{Exact convergence rates to derivatives of local time for some self-similar Gaussian processes}
\author{Minhao Hong}
\usepackage{amsmath,amssymb,amsthm}

\allowdisplaybreaks[4]
\newtheorem{theorem}{Theorem}[section]
\newtheorem{lemma}[theorem]{Lemma}
\newtheorem{corollary}[theorem]{Corollary}
\newtheorem{proposition}[theorem]{Proposition}
\newtheorem{remark}[theorem]{Remark}
\numberwithin{equation}{section}
\numberwithin{figure}{section}
\usepackage[
bookmarks=true,         
bookmarksnumbered=true, 
colorlinks=true, pdfstartview=FitV, linkcolor=blue, citecolor=blue,
urlcolor=blue]{hyperref}

\def\R{\mathbb{R}}
\def\E{{{\mathbb E}\,}}
\def\C{\mathbb{C}}

\def\Var{{\mathop {{\rm Var\, }}}}
\def\Cov{{\mathop {{\rm Cov\, }}}}
\def\mS{\mathcal{S}}
\def\display{\displaystyle}
\begin{document}
	\maketitle
	
	\begin{abstract}
		In this article, for some $d-$dimensional Gaussian processes 
		\[X=\big\{X_t=(X^1_t,\cdots,X^d_t):t\ge0\big\},\]
		whose components are i.i.d. $1-$dimensional self-similar Gaussian process with Hurst index $H\in(0,1)$, we consider the asymptotic behavior of  approximation of its $\boldsymbol{k}-$th derivatives of local time under certain mild conditions, where 
		$\boldsymbol{k}=(k_1,\cdots,k_d)$ and $k_\ell$'s are non-negative real numbers.  We will give a derivative version of the limit theorems for functional of Gaussian processes and use this result to get the asymptotic behaviors.
	\end{abstract}
	\section{Introduction}
	
	For $H\in(0,1)$, consider $G^{H}_d$, the set of $d$-dimensional Gaussian processes $$X=\{X_t=(X^1_t,\cdots,X^d_t):t\ge0\}$$ whose components are i.i.d. $1-$dimensional self-similar Gaussian processes with Hurst index $H\in(0,1)$ and satisfying the following local non-determinism:
	
	{\bf (LND)}: Given $n\ge1$, there exists a positive constant $\kappa_{H,n}$ s.t. 
	\begin{align}\label{local}
		\Var\Big(\sum_{j=1}^{n}x_j\big(X^1_{t_j}-X^1_{t_{j-1}}\big)\Big)\ge \kappa_{H,n}\sum_{j=1}^{n}|x_j|^2\big(t_j-t_{j-1}\big)^{2H}
	\end{align}
	for any $x_1,x_2,\cdots,x_n\in\R$ and  $0=t_0<t_1<t_2<\cdots<t_n<\infty$.
	
	 Let $T>0$, $x\in\R^d$ and $\boldsymbol{k}=(k_1,\cdots,k_d)$ with $k_1,\cdots,k_d\ge0$, we denote $L^{(\boldsymbol{k})}(T,x)$, the $\boldsymbol{k}-$th order derivatives of local time for $X\in G^{H}_d$, by the limit of 
	\begin{align}\label{higherorder}
		L_{\varepsilon}^{(\boldsymbol{k})}(T,x)=\int_{0}^{T}p_{\varepsilon}^{(\boldsymbol{k})}(X_t+x)dt
	\end{align} 
	as $\varepsilon\downarrow0$ in $L^p(p\ge1)$, where $\displaystyle p_\varepsilon(x)=\frac{1}{(2\pi\varepsilon)^{\frac{d}{2}}}\exp\Big(-\frac{1}{2\varepsilon}|x|^2\Big)$ is the heat kernel function in $\R^d$ and $p^{(\boldsymbol{k})}_\varepsilon(x)$ is the $\boldsymbol{k}-$th derivatives of $p_\varepsilon(x)$. In particular, when $\boldsymbol{k}=\boldsymbol{0}$, we have $\displaystyle p_\varepsilon^{(\boldsymbol{0})}(x)=p_\varepsilon(x)$ and $\displaystyle L(T,x)=L_{\varepsilon}^{(\boldsymbol{0})}(T,x)$ is the local time for $X$. Local time and its derivatives have been researched by many authors, see \cite{no2007}, \cite{wx2010}, \cite{l2022}, \cite{y2014}, \cite{ghx2019}, \cite{hx2020} and references there in. 
	
	An important question for local time and its derivatives is the asymptotic behavior of approximation $L_{\varepsilon}^{(\boldsymbol{k})}(T,x)$. In fact, heat kernel $p_\varepsilon(x)$ is also used in the construction of approximation of self-intersection local time or its derivatives, whose asymptotic behavior  has been researched by \cite{m2008}, \cite{jn2017} and \cite{jn2019} (We recommend \cite{r1987}, \cite{hn2005}, \cite{y2020}, \cite{yy2015} and references there in for self-intersection local time and derivatives of local time). It is natural to consider the exact rate of convergence to local time or derivatives of local time. In early age the convergence rate of local time's definition $\display L^x_T:=\lim\limits_{\varepsilon\downarrow0}\frac{1}{\varepsilon}\int_{0}^{T}\mathbf{1}_{[x,x+\varepsilon]}\big(X_t\big)dt$
	has been discussed for Brownian motion in \cite{k1994} and for symmetric L\'evy processes in \cite{mr1994}, where the limit theorems exists with probability $1$. 
	
	In this article, we will discuss the exact convergence rate of $L_{\varepsilon}^{(\boldsymbol{k})}(T,x)$ for $G^H_d$. From Remark 1.4 in \cite{hx2020}, $L_{\varepsilon}^{(\boldsymbol{k})}(T,x)$ converges to $L^{(\boldsymbol{k})}(T,x)$ in $L^p(p\ge1)$ when $H(2|\boldsymbol{k}|+d)<1$, where $\display|\boldsymbol{k}|=\sum_{\ell=1}^{d}k_\ell$. To get a more accurate asymptotic property, given $\sigma>0$, $d\ge1$ and $H\in(0,1)$, we introduce $G^{\sigma,H}_d$,  the set of $d$-dimensional Gaussian processes $X=\{X_t=(X^1_t,\cdots,X^d_t):t\ge0\}$ whose components are i.i.d. $1-$dimensional self-similar Gaussian processes with Hurst index $H$ and satisfying the following properties:
	
	{\bf (A)} (Strong local non-determinism): $\{X^1_t:t\ge0\}$ satisfies the following strong local non-determinism property: there exists a constant $\kappa_H>0$ s.t. for any integer $m\ge1$ and any times $0=s_0<s_1\le s_2\le\cdots s_m<t<\infty$, 
	\begin{align}\label{slocal}
		\Var\Big(X^1_t\big|X^1_{s_1},\cdots,X^1_{s_m}\Big)\ge\kappa_H\min_{1\le j\le m}\big|t-s_{j}\big|^{2H}.
	\end{align}
	
	{\bf (B)}: There exists a constant $\eta>0$ and a non-negative increasing function $\phi(\varepsilon)$ on $[0,1/\eta]$ with $\lim\limits_{\varepsilon\downarrow0}\phi(\varepsilon)=0$, such that 
	\begin{align}\label{prob}
		0\le h^{2H}\big(\sigma-\phi(h/t)\big)\le\Var\Big(X^1_{t+h}-X^1_{t}\Big)\le h^{2H}\big(\sigma+\phi(h/t)\big)
	\end{align}
	for any $t\ge0$ and $h\in[0,1/\eta]$.
	
	{\bf (C):}  There exists a non-negative decreasing function $\psi(\eta)$ on $[1,\infty)$ with $\lim\limits_{\eta\to\infty}\psi(\eta)=0$ such that 
	\begin{align}\label{proc}
		\Big|\E\big(X^1_{t_4}-X^1_{t_3}\big)\big(X^1_{t_2}-X^1_{t_1}\big)\Big|\le\psi(\eta)\big(t_4-t_3\big)^H\big(t_2-t_1\big)^H
	\end{align}
	for any $0\le t_1\le t_2 \le t_3\le t_4<\infty$ and $\eta>1$ satisfying (i) $\frac{\Delta t_2}{\Delta t_4}\le\frac{1}{\eta}$ or (ii) $\frac{\Delta t_2}{\Delta t_4}\ge\eta$ or (iii) $\frac{\max\{\Delta t_2,\Delta t_4\}}{\Delta t_3}\le\frac{1}{\eta}$, where $\Delta t_i=t_i-t_{i-1}$, $i=2,3,4$.
	\begin{remark}
		Theorem 2.6 in \cite{n1989} shows that for a Gaussian process, we can get local non-determinism (\ref{local}) from Strong local non-determinism (\ref{slocal}) so that $G^{\sigma,H}_d$ is a subset of $G^{H}_d$. From Lemma 2.4--2.6 in \cite{sxy2019} and Corollary 5.3--5.4 in \cite{hlx2023}, we get that self-similar Gaussian processes including sub-fractional Brownian motion (sfBm) and bi-fractional Brownian motion (bi-fBm) belong to $G^{\sigma,H}_d$.
	\end{remark}
	We can use limit theorems for functional of $X$ in \cite{hlx2023} to get the asymptotic behavior of $L_{\varepsilon}(T,x)$. Using the definition of $p_{\varepsilon}$, we obtain 
		\begin{align}\label{functional1}
		L_{\varepsilon}(T,x)=\int_{0}^{T}p_{\varepsilon}(X_t+x)dt=\varepsilon^{-\frac{d}{2}}\int_{0}^{T}p_{1}\big(\varepsilon^{-\frac12}(X_t+x)\big)dt,
	\end{align} 
	where $p_1$ belongs to function space 
	\begin{align}\label{Theta2}
		\Theta^2=\Big\{f\in L^2(\R^d):\int_{\R^d}|f(u)||u|^2 dx<\infty,\int_{\R^d}u f(u)dx=0\Big\}
	\end{align} 
	defined in \cite{hlx2023}. Using the notation $\displaystyle\widehat{f}(x)=\int_{\R^d}f(u)e^{-\iota x\cdot u}du$, where "$\cdot$" is the inner product in Euclidean space $\R^d$, Theorem 1.1--1.3 in \cite{hlx2023} tell us that for any $f\in\Theta^2$, $x\in\R^d$ and $X\in G^{\sigma,H}_d$, there exist:
	
	(1): When $H(4+d)<1$,
		\begin{align*}
		n^{H\beta} \Big(n^H\int^T_0 f(n^H(X_s-x))ds&-L(T,-x)\int_{\mathbb{R}^d} f(u)du\Big)\\&\overset{L^p}{\longrightarrow} \int_{\mathbb{R}^d} \Big(- \frac{1}{2}\sum\limits^d_{i,k=1} u_i L^{({\bf e}_i+{\bf e}_k)}(T,-x)u_k\Big) f(u) du
	\end{align*}
	for any $p\geq 1$ as $n$ tends to infinity, where ${\bf e}_i$ is the $i$-th unit vector in $\mathbb{R}^d$.
		
	(2): When $H(4+d)=1$,
	\begin{align*}
		& \left\{(\ln n)^{-\frac{1}{2}} n^{\frac{Hd+1}{2}} \Big(\int^T_0 f(n^H(X_s-x))ds-n^{-Hd}L(T,-x)\int_{\mathbb{R}^d} f(u)du\Big),\; 
		T\ge 0\right\} \\
		&\qquad\qquad\qquad\qquad\qquad\qquad\qquad\qquad\qquad\qquad\qquad \overset{law}{\longrightarrow } \ \
		\Big\{  \sqrt{D_{H,d,f}} W\big(L(T,-x)\big) \,, T\ge
		0\Big\}
	\end{align*}
	in the space $C([0,\infty ))$ as $n$ tends to infinity,  where  $W $ is a real-valued standard Brownian
	motion independent of $X$ and
	\begin{align}\label{DHdf}
		D_{H,d, f}   \nonumber
		&=\frac{2}{(2\pi)^d} \lim_{n\to\infty}\frac{1}{\ln n}   \int_0^{n} \int_{\mathbb{R}^d}\Big|\widehat{f}(y)-\widehat{f}(0)\Big|^2e^{-\frac{\sigma}{2}|y|^2s^{2H}}dyds \\ 
		&=\frac{2}{(2\pi)^d \sigma^{\frac{1}{2H}}} \int_{\mathbb{R}^d}\Big| \int_{\mathbb{R}^d} f(z) \frac{1}{2}(z\cdot u)^{2} dz\Big|^2e^{-\frac{1}{2}|u|^2}du.
	\end{align}
	
	(3): When $H(4+d)>1$ and $Hd<1$, 
	\begin{align*}
		& \left\{ n^{\frac{Hd+1}{2}} \Big(\int^T_0 f(n^H(X_s-x))ds-n^{-Hd}L(T,-x)\int_{\mathbb{R}^d} f(u)du\Big),\; 
		T\ge 0\right\} \\&\qquad\qquad\qquad\qquad\qquad\qquad\qquad\qquad\qquad\qquad\overset{law}{\longrightarrow }  \left\{ \sqrt{C_{H,d,f}} W\big(L(T,-x)\big) \,, T\ge
		0\right\}
	\end{align*}
	in the space $C([0,\infty ))$ as $n$ tends to infinity, where $W $ is a real-valued standard Brownian
	motion independent of $X$ and 
	\begin{align*}
		C_{H,d, f} 
		&=\frac{2}{(2\pi)^d}   \int_0^\infty \int_{\mathbb{R}^d}\Big|\widehat{f}(y)-\widehat{f}(0)\Big|^2e^{-\frac{\sigma}{2}|y|^2 s^{2H}}dyds.
	\end{align*}

	Then it is easy to get the following results of $L_\varepsilon(T,x)$:
	
	(1): When $H(4+d)<1$, 
	\begin{align*}
		\varepsilon^{-1}\Big(L_\varepsilon(T,x)-L(T,x)\Big)
		\stackrel{L^p}{\longrightarrow}-\frac{1}{2}\sum_{i=1}^{d}L^{(2\boldsymbol{e}_i)}(T,x)
	\end{align*}
	as $\varepsilon\downarrow0$, where $\boldsymbol{e}_{i}$ are the $i-$th unit vector in $\mathbb{R}^d$ for $i=1,2,\cdots,d$.
	
	(2): When $H(4+d)=1$,
	\begin{align*}
		\Big\{\ln^{-\frac12}(1+\varepsilon^{-\frac12})\varepsilon^{-1}\Big(L_\varepsilon(T,x)-L(T,x)\Big):T>0\Big\}\stackrel{law}{\longrightarrow}\Big\{\sqrt{D_{H,d,p_1}}W\big(L(T,x)\big):T>0\Big\}
	\end{align*}
	in space $C([0,\infty))$ as $\varepsilon\downarrow0$, where $W$ is a real-valued standard Brownian motion independent of $X$ and 
	\[D_{H,d,p_1}=\frac{1}{2H(2\pi)^d\sigma^{\frac{1}{2H}}}\int_{\R^d}|x|^4e^{-\frac{1}{2}|x|^2}dx.\]
	
	(3): When $H(4+d)>1$ and $Hd<1$,
	\begin{align*}
		\Big\{\varepsilon^{-\frac{1-Hd}{4H}}\Big(L_\varepsilon(T,x)-L(T,x)\Big):T>0\Big\}\stackrel{law}{\longrightarrow}\Big\{\sqrt{C_{H,d,p_1}}W\big(L(T,x)\big):T>0\Big\}
	\end{align*}
	in space $C([0,\infty))$ as $\varepsilon\downarrow0$, where $W$ is a real-valued standard Brownian motion independent of $X$ and 
	\[C_{H,d,p_1}=\frac{2}{(2\pi)^{d}\sigma^{\frac{1}{2H}}}\int_{0}^{\infty}\int_{\R^d}|e^{-\frac{1}{2}|y|^2}-1|^2e^{-\frac{1}{2}|y|^2s^{2H}}dyds.\]

However, when $|\boldsymbol{k}|>0$, we can not apply Theorem 1.1--1.3 of \cite{hlx2023} to finding the convergence rate of $L^{(\boldsymbol{k})}_{\varepsilon}(T,x)$ to $L^{(\boldsymbol{k})}(T,x)$. Now
\begin{align}\label{functional2}
	L_{\varepsilon}^{(\boldsymbol{k})}(T,x)=\int_{0}^{T}p_{\varepsilon}^{(\boldsymbol{k})}(X_t+x)dt=\varepsilon^{-\frac{|\mathbf{k}|+d}{2}}\int_{0}^{T}p_{1}^{(\boldsymbol{k})}\big(\varepsilon^{-\frac12}(X_t+x)\big)dt
\end{align}  
and from Proposition 3.6 in \cite{h2016},  $\displaystyle\int_{\R^d}p_1^{(\boldsymbol{k})}(x)dx=0$ so that this asymptotic behavior belongs to the "zero energy" case (see \cite{hnx2014}, \cite{nx2013}, \cite{nx2014}, \cite{nx2019} and \cite{jnp2021} for more information).

The main work of this article is to extend limit theorems for functional of $X$ when $f$ is derivative of another function and use the theorems to get the exact rate of convergence of $L_{\varepsilon}^{(\boldsymbol{k})}(T,x)$ to  $L^{(\boldsymbol{k})}(T,x)$ when $|\boldsymbol{k}|>0$. Note that $p_1(x)$ belongs to the Schwartz space $\mS(\R^d)$ of rapidly decreasing functions,  which includes all $C^{\infty}$ function $f:\R^d\to\C$ such that given $\boldsymbol{k}=(k_1,\cdots,k_d)$ with non-negative integers $k_1,\cdots,k_d$ and non-negative integer $N$, there exists a positive number $C_{\boldsymbol{k},N}$, s.t.
\[|f^{(\boldsymbol{k})}(x)|\le C_{\boldsymbol{k},N}\big(1+|x|^2\big)^{-N}\] 
for all $x\in\R^d$, where $|z|$ is the Euclidean norm of complex number if $z\in\C$. Given $N\ge1$ and $d\ge1$, we denote $\widetilde{\mS}_{d,1}={\mS}(\R^d)$ 
and   
\begin{align}
	\widetilde{\mS}_{d,N}=\Big\{f\in\widetilde{\mS}_{d,N-1}:\int_{\R^d}x_{\ell_1}x_{\ell_2}\cdots x_{\ell_{N-1}} f(x)dx=0,\ell_j=1,2,\cdots,d\textit{ for any }j=1,2\cdots,N-1\Big\}.
\end{align}
for $N=2,3,\cdots$. The following are main results of this article:
\begin{theorem}\label{thmlp}
	Assume $X\in G^{H}_{d}$ for some $d\ge1$ and $H\in(0,1)$. Given any integer $N\ge1$ and $d-$dimensional vector $\boldsymbol{k}=(k_1,\cdots,k_d)$ with non-negative integers $k_1,\cdots,k_d$, When $H(2|\boldsymbol{k}|+d)<1-2NH$,
	\begin{align*}
		\varepsilon^{-\frac{N}{2}}\bigg(\varepsilon^{-\frac{|\boldsymbol{k}|+d}{2}}\int_{0}^{T}f^{(\boldsymbol{k})}&\big(\varepsilon^{-\frac12}(X_t+x)\big)dt-\widehat{f}(0)L^{(\boldsymbol{k})}(T,x)\bigg)\\&\stackrel{L^p}{\longrightarrow}\frac{\iota^{N}}{N!}\sum_{i_1,\cdots,i_{N}=1}^{d}L^{(\boldsymbol{k}+\boldsymbol{e}_{i_1}+\boldsymbol{e}_{i_2}+\cdots+\boldsymbol{e}_{i_{N}})}(T,x)\int_{\R^d}v_{i_1}v_{i_2}\cdots v_{i_{N}}f(v)dv
	\end{align*}
	for any $f\in\widetilde{\mS}_{d,N}$, where $\boldsymbol{e}_{i}$ is the $i-$th unit vector in $\mathbb{R}^d$ for $i=1,2,\cdots,d$ and $\iota=\sqrt{-1}$.
\end{theorem}
\begin{theorem}\label{thmdis}
	Assume $X\in G^{\sigma,H}_{d}$ for some $d\ge1$ and $H\in(0,1)$. Given any integer $N\ge1$ and $d-$dimensional vector $\boldsymbol{k}=(k_1,\cdots,k_d)$ with non-negative integers $k_1,\cdots,k_d$, when $H(2|\boldsymbol{k}|+d)\ge1-2NH$ and $H(2|\boldsymbol{k}|+d)<1$, 
	 \begin{align*}
	\ell_{\varepsilon,H,d}^{(|\boldsymbol{k}|,N)}\bigg(\varepsilon^{-\frac{|\boldsymbol{k}|+d}{2}}\int_{0}^{T}f^{(\boldsymbol{k})}\big(\varepsilon^{-\frac12}(X_t+x)\big)dt&-\widehat{f}(0)L^{(\boldsymbol{k})}(T,x)\bigg)\stackrel{law}{\longrightarrow}\Big\{\sqrt{D_{H,d}}W\big(L(T,x)\big):T>0\Big\}
	\end{align*}
	in space $C([0,\infty))$ as $\varepsilon\downarrow0$ for any $f\in\widetilde{\mS}_{d,N}$, where $W$ is a real-valued standard Brownian motion independent of $X$,
	\begin{align*}
		\ell_{\varepsilon,H,d}^{(|\boldsymbol{k}|,N)}=\left\{\begin{array}{ll}
			\varepsilon^{-\frac{N}{2}}\ln^{-\frac12}(1+\varepsilon^{-\frac12}), &\textit{if } 1-2NH=H(2|\boldsymbol{k}|+d);\\
			\varepsilon^{\frac{1}{4}(2|\boldsymbol{k}|+d-\frac{1}{H})}, &\textit{if } 1-2NH<H(2|\boldsymbol{k}|+d)<1;
		\end{array}\right.
	\end{align*}
	and
		\begin{align}\label{Dhd}
			D_{H,d}=\left\{\begin{array}{ll}
				\displaystyle\frac{2}{H(2\pi)^d\sigma^{\frac{1}{2H}}}\int_{\R^d}\Big|\int_{\R^d}\frac{f(z)\big(z\cdot x\big)^{N}}{N!}dz\Big|^2\Big(\prod_{\ell=1}^{d}|x_\ell|^{2k_\ell}\Big)e^{-\frac{1}{2}|x|^2}dx,&1-2NH=H(2|\boldsymbol{k}|+d);\\\displaystyle\frac{2}{(2\pi)^d\sigma^{\frac{1}{2H}}}\int_{0}^{\infty}\int_{\R^d}\big|\widehat{f}(x)-\widehat{f}(0)\big|^2\Big(\prod_{\ell=1}^{d}|x_\ell|^{2k_\ell}\Big)e^{-\frac{1}{2}|x|^2s^{2H}}dxds,&1-2NH<H(2|\boldsymbol{k}|+d)<1.
			\end{array}\right.
		\end{align}
\end{theorem}
\begin{remark}
	The results in Theorem \ref{thmlp} and Theorem \ref{thmdis} can be extended to the case $k_1,k_2,\cdots,k_d$ are non-negative real numbers. Here the partial derivative $f^{(\boldsymbol{k})}(x)$ is defined by 
\begin{align}\label{multideri}
	f^{(\boldsymbol{k})}(x)=D^{k_1}_{x_1,+}\cdots D^{k_d}_{x_d,+}f(x_1,\cdots,x_d),
\end{align}
where given $k_\ell\ge0$ and $g:\R\to\C$, $D^{k_\ell}_{x,+}g(x)$ is the common $k_\ell-$th derivative of $g(x)$ when $k_\ell$ is integer $\big(D^{0}_{x,+}g(x)=g(x)\big)$ and  \[D^{k_\ell}_{x,+}g(x)=\frac{d^n}{dx^n}D^{\{k_\ell\}}_{x,+}g(x)=\frac{\{k_\ell\}}{\Gamma(1-\{k_\ell\})}\int_{0}^{\infty}\frac{g^{(n)}(x)-g^{(n)}(x-t)}{t^{1+\{k_\ell\}}}dt\]
for $n$ the largest integer no greater than $k_\ell$ and $\{k_\ell\}=k_\ell-n$ when $k_\ell$ is not integer. The definition of fractional derivatives comes from Section 5.6 of \cite{skm1993} and in this case 
\[L^{(\boldsymbol{k})}_{\varepsilon}(T,x)=\frac{1}{(2\pi)^d}\int_{0}^{T}\int_{\mathbb{R}^d}\prod_{\ell=1}^{d}\big(\iota u_{\ell})^{k_\ell}e^{-\frac{1}{2}\varepsilon|u|^2}\times\exp\big(\iota(X_t+x)\cdot u\big)dudt\]
from Lemma 5.1 in \cite{hy2024}, where we use convenient notation 
\begin{align}\label{fracpower}
	\big(\iota x\big)^k=\left\{\begin{array}{cl}\iota^kx^k,&k\text{ is an integer};\\|x|^ke^{\iota\frac{\pi k}{2}\mathrm{sgn}(x)},&k\text{ is not an integer}.\end{array}\right.
\end{align}
\end{remark}

Note that $p_1(x)\in\widetilde{\mS}_{d,2}$. Using Theorem \ref{thmlp} and \ref{thmdis}, we can get the following corollary easily:
\begin{corollary}
	Suppose $X\in G^{\sigma,H}_{d}$ and $\boldsymbol{k}=(k_1,\cdots,k_d)$, where $k_1,\cdots,k_d$ are non-negative integers. The following results exist:
	
	{\bf (1):} When $H(2|\boldsymbol{k}|+d)<1-4H$,
	\begin{align*}
		\varepsilon^{-1}\bigg(L^{(\boldsymbol{k})}_{\varepsilon}(T,x)-L^{(\boldsymbol{k})}(T,x)\bigg)\stackrel{L^p}{\longrightarrow}-\frac{1}{2}\sum_{i,k=1}^{d}L^{(\boldsymbol{k}+\boldsymbol{e}_i+\boldsymbol{e}_k)}(T,x),
	\end{align*}
	where $\boldsymbol{e}_{i}$ are the $i-$th unit vector in $\mathbb{R}^d$ for $i=1,2,\cdots,d$.	 
	
		{\bf (2):} When $H(2|\boldsymbol{k}|+d)=1-4H$,
	\begin{align*}
		\varepsilon^{-1}\ln^{-\frac12}(1+\varepsilon^{-\frac12})\bigg(L^{(\boldsymbol{k})}_{\varepsilon}(T,x)-L^{(\boldsymbol{k})}(T,x)\bigg)\stackrel{law}{\longrightarrow}\Big\{\sqrt{\widetilde{D}_{H,d,1}}W\big(L(T,x)\big):T>0\Big\},
	\end{align*}
	in space $C([0,\infty))$ as $\varepsilon\downarrow0$, where $W$ is a real-valued standard Brownian motion independent of $X$ and 
	\[\widetilde{D}_{H,d,1}=\frac{1}{2H(2\pi)^d\sigma^{\frac{1}{2H}}}\int_{\R^d}|x|^{4}\Big(\prod_{\ell=1}^{d}|x_\ell|^{2k_\ell}\Big)e^{-\frac{1}{2}|x|^2}dx.\]
	
	{\bf (3):} When $1-4H<H(2|\boldsymbol{k}|+d)<1$, 
	\begin{align*}
		\varepsilon^{\frac{1}{4}(2|\boldsymbol{k}|+d-\frac{1}{H})}\bigg(L^{(\boldsymbol{k})}_{\varepsilon}(T,x)-L^{(\boldsymbol{k})}(T,x)\bigg)\stackrel{law}{\longrightarrow}\Big\{\sqrt{\widetilde{D}_{H,d,2}}W\big(L(T,x)\big):T>0\Big\},
	\end{align*}
	in space $C([0,\infty))$ as $\varepsilon\downarrow0$, where $W$ is a real-valued standard Brownian motion independent of $X$ and 
	\[\widetilde{D}_{H,d,2}=\frac{2}{(2\pi)^d\sigma^{\frac{1}{2H}}}\int_{0}^{\infty}\int_{\R^d}\big|e^{-\frac{1}{2}|x|^2}-1\big|^2\Big(\prod_{\ell=1}^{d}|x_\ell|^{2k_\ell}\Big)e^{-\frac{1}{2}|x|^2s^{2H}}dxds.\]
\end{corollary}

After some preliminaries in Section $2$, we prove Theorem \ref{thmlp} in Section $3$ and Theorem \ref{thmlp} in Section $4$ and In Section $5$ we prove some technical lemmas. Throughout this paper, if not mentioned otherwise, the letter $c$, 
with or without a subscript or superscript, denotes a generic positive finite constant whose exact value may change from line to line. Moreover, denote $x\wedge{y}=\min\{x,y\}$ for any $x,y\in\R$ in this article.
\section{Preliminaries}
In this section, we introduce some lemmas which will contribute to the proof of Theorem \ref{thmlp} and Theorem \ref{thmdis}. First we will prove a lemma about $f\in \widetilde{\mS}_{d,N}$, 
\begin{lemma}\label{2lma1}
	Given $N\ge1$, assume $f\in\widetilde{\mS}_{d,N}$. There exists a positive constant $C_{f}$ s.t. 
	\begin{align*}
		\big|\widehat{f}(x+y)-\widehat{f}(x)\big|\le C_{f}\Big(\big(|x|^{N-1}|y|\big)\wedge1+|y|^N\wedge1\Big)
	\end{align*}
	for any $x,y\in\R^d$. Particularly, $\display\big|\widehat{f}(y)-\widehat{f}(0)\big|\le C_{f}(|y|\wedge1)^N$.
\end{lemma}
\begin{proof}
	Because $f\in\widetilde{\mS}_{d,N}$, we have 
	\begin{align*}
		\big|\widehat{f}(x+y)-\widehat{f}(x)\big|&=\Big|\int_{\R^d}f(u)e^{\iota u\cdot(x+y)}du-\int_{\R^d}f(u)e^{\iota u\cdot x}du\Big|\\&=\Big|\int_{\R^d}f(u)e^{\iota u\cdot x}\big(e^{\iota u\cdot y}-1\big)du\Big|\\&\le\Big|\int_{\R^d}f(u)e^{\iota u\cdot x}\Big(e^{\iota u\cdot y}-\sum_{j=0}^{N-1}\frac{(u\cdot y)^j}{j!}\Big)du\Big|+\Big|\int_{\R^d}f(u)\big(e^{\iota u\cdot x}-1\big)\sum_{j=1}^{N-1}\frac{(u\cdot y)^j}{j!}du\Big|\\&\le C_{1,1}\Big(|y|^N+|x|\sum_{j=1}^{N-1}|y|^j\Big)\le C_{1,2}\Big(|x|^{N-1}|y|+|y|^N\Big).
	\end{align*}
	Note that $\display\big|\widehat{f}(x)\big|=\Big|\int_{\R^d}f(u)e^{\iota u\cdot x}du\Big|\le\int_{\R^d}|f(u)|du<\infty$ for any $x\in\R^d$. We can get the desired result from inequality $(|a|+|b|)\wedge1\le2(|a|\wedge1+|b|\wedge1)$ for any $a,b\in\R^d$.
\end{proof}
Using Lemma \ref{2lma1}, we can get a further conclusion:
\begin{lemma}\label{2lma2}
	Assume $N\ge1$, $d\ge1$ and $\boldsymbol{k}=(k_1,\cdots,k_d)$ with non-negative integers $k_1,k_2,\cdots,k_d$. If $f\in\widetilde{\mS}_{d,N}$, when $|\boldsymbol{k}|>0$ there exists a positive constant $C_{n,\boldsymbol{k},f}$, s.t. 
	\begin{align*}
		&\bigg|\prod_{j=1}^{n}\prod_{\ell=1}^{d}(\iota x_{j\ell}+\iota y_{j\ell})^{k_\ell}\Big(\widehat{f}(x_j+y_j)-\widehat{f}(0)\Big)-\prod_{j=1}^{n}\prod_{\ell=1}^{d}(\iota x_{j\ell})^{k_\ell}\Big(\widehat{f}(x_j)-\widehat{f}(0)\Big)\bigg|\\\le& C_{n,\boldsymbol{k},f}\sum_{j=1}^{n}\bigg(|x_j|^{|\boldsymbol{k}|}\big(|x_j|^{N-1}|y_j|\big)\wedge1+|y_j|^{|\boldsymbol{k}|}(|y_j|\wedge1)^{N}+|x_j|^{|\boldsymbol{k}|-1}(|x_j|\wedge1)^{N}|y_j|+|y_j|^{|\boldsymbol{k}|}(|x_j|\wedge1)^{N}\bigg)\\&\qquad\qquad\qquad\qquad\qquad\qquad\qquad\qquad\qquad\qquad\qquad\qquad\times\prod_{i\ne j}\Big(|x_i|^{|\boldsymbol{k}|}(|x_i|\wedge1)^{N}+|y_i|^{|\boldsymbol{k}|}(|y_i|\wedge1)^{N}\Big)
	\end{align*}
 and when $|\boldsymbol{k}|=0$ there exists a positive constant $C_{n,f}$, s.t. 
\begin{align*}
	\bigg|\prod_{j=1}^{n}\Big(\widehat{f}(x_j+y_j)-\widehat{f}(0)\Big)&-\prod_{j=1}^{n}\Big(\widehat{f}(x_j)-\widehat{f}(0)\Big)\bigg|\\&\le C_{n,f}\sum_{j=1}^{n}\Big(\big(|x_j|^{N-1}|y_j|\big)\wedge1+|y_j|^N\wedge1\Big)\prod_{i\ne j}\Big(|y_i|^N\wedge1+|x_i|^N\wedge1\Big),
\end{align*}
	where $x_j=(x_{j1},\cdots,x_{jd})\in\R^d$ and $y_j=(y_{j1},\cdots,y_{jd})\in\R^d$ for $j=1,2,\cdots,n$.
\end{lemma}
\begin{proof}
	When $|\boldsymbol{k}|>0$, we have $|\boldsymbol{k}|\ge1$ and $k_\ell\ge1$ for some $\ell$. Recall that $k_1,\cdots,k_d$ are non-negative integers, 
	\begin{align*}
		\Big|\prod_{\ell=1}^{d}(\iota x_{j\ell}+\iota y_{j\ell})^{k_\ell}-\prod_{\ell=1}^{d}(\iota x_{j\ell})^{k_\ell}\Big|&\le\sum_{\ell=1,k_\ell>0}^{d}\big|(\iota x_{j\ell}+\iota y_{j\ell})^{k_\ell}-(\iota x_{j\ell})^{k_\ell}\big|\prod_{\ell'\ne\ell}\Big(\big|\iota x_{j\ell'}+\iota y_{j\ell'}\big|^{k_{\ell'}}+\big|\iota x_{j\ell'}\big|^{k_{\ell'}}\Big)\\&\le C_{1,3}\sum_{\ell=1,k_\ell>0}^{d}\Big(\sum_{m_\ell=1}^{k_\ell}|x_{j\ell}|^{k_\ell-m_\ell}|y_{j\ell}|^{m_\ell}\Big)\prod_{\ell'\ne\ell}\Big(\max\big\{|x_{j\ell'}|,|y_{j\ell'}|\big\}\Big)^{{k_{\ell'}}}\\&\le C_{1,4}\sum_{\ell=1,k_\ell>0}^{d}|y_{j\ell}|\Big(\max\big\{|x_{j\ell}|,|y_{j\ell}|\big\}\Big)^{{k_{\ell}-1}}\prod_{\ell'\ne\ell}\Big(\max\big\{|x_{j\ell'}|,|y_{j\ell'}|\big\}\Big)^{^{k_{\ell'}}}\\&\le C_{1,5}\Big(\max\big\{|x_{j}|,|y_{j}|\big\}\Big)^{{|\boldsymbol{k}|-1}}|y_j|.
	\end{align*}
	Then combining this inequality with Lemma \ref{2lma1} we obtain
	\begin{align*}
		&\bigg|\prod_{\ell=1}^{d}(\iota x_{j\ell}+\iota y_{j\ell})^{k_\ell}\Big(\widehat{f}(x_j+y_j)-\widehat{f}(0)\Big)-\prod_{\ell=1}^{d}(\iota x_{j\ell})^{k_\ell}\Big(\widehat{f}(x_j)-\widehat{f}(0)\Big)\bigg|\\\le&\Big|\prod_{\ell=1}^{d}(\iota x_{j\ell}+\iota y_{j\ell})^{k_\ell}\Big|\cdot\Big|\widehat{f}(x_j+y_j)-\widehat{f}(x_j)\Big|+\Big|\prod_{\ell=1}^{d}(\iota x_{j\ell}+\iota y_{j\ell})^{k_\ell}-\prod_{\ell=1}^{d}(\iota x_{j\ell})^{k_\ell}\Big|\cdot\Big|\widehat{f}(x_j)-\widehat{f}(0)\Big|\\\le&C_{1,6}\bigg(\Big(\max\big\{|x_{j}|,|y_{j}|\big\}\Big)^{{|\boldsymbol{k}|}}\Big(\big(|x_j|^{N-1}|y_j|\big)\wedge1+|y_j|^N\wedge1\Big)\\&\qquad\qquad\qquad\qquad\qquad\qquad\qquad\qquad+\Big(\max\big\{|x_{j}|,|y_{j}|\big\}\Big)^{{|\boldsymbol{k}|}-1}|y_j|\big(|x_j|\wedge1\big)^N\bigg)\\\le&C_{1,7}\bigg(|x_j|^{|\boldsymbol{k}|}\big(|x_j|^{N-1}|y_j|\big)\wedge1+|y_j|^{|\boldsymbol{k}|}(|y_j|\wedge1)^{N}+|x_j|^{|\boldsymbol{k}|-1}(|x_j|\wedge1)^{N}|y_j|+|y_j|^{|\boldsymbol{k}|}(|x_j|\wedge1)^{N}\bigg).
	\end{align*}
	The proof of $|\boldsymbol{k}|>0$ is finished by the inequality
	\begin{align*}
		&\bigg|\prod_{\ell=1}^{d}(\iota x_{j\ell}+\iota y_{j\ell})^{k_\ell}\Big(\widehat{f}(x_j+y_j)-\widehat{f}(0)\Big)-\prod_{\ell=1}^{d}(\iota x_{j\ell})^{k_\ell}\Big(\widehat{f}(x_j)-\widehat{f}(0)\Big)\bigg|\\\le&C_{1,8}\Big(|x_j-y_j|^{|\boldsymbol{k}|}(|x_j-y_j|\wedge1)^{N}+|x_j|^{|\boldsymbol{k}|}(|x_j|\wedge1)^{N}\Big)\le C_{1,9}\Big(|x_j|^{|\boldsymbol{k}|}(|x_j|\wedge1)^{N}+|y_j|^{|\boldsymbol{k}|}(|y_j|\wedge1)^{N}\Big),
	\end{align*}
	which comes from Lemma \ref{2lma1}.
When $|\boldsymbol{k}|=0$ the proof can be obtained by following the similar way, which completes the proof.
\end{proof}
For the functions $|x|^{a}$ and $(|x|\wedge1)^b$ appearing in Lemma \ref{2lma2}, we have the following result:
\begin{lemma}\label{2lma3}
	For any $T>0$, $H>0$, $\gamma\ge1$, $a\ge0$, $b\ge0$ and positive integer $d$, there exists a positive constant $C_{H,d}^{a,b}$ depending on $H,d$, such that 
	\begin{align*}
		\int_{0}^{T}\int_{\R^d}|x|^a\Big(\big|\frac{1}{\gamma}\varepsilon^{\frac{1}{2}}x\big|^b\wedge1\Big)e^{-|x|^2t^{2H}}dxdt\le \frac{C_{H,d}^{a,b}}{\gamma^{(\frac{1}{H}-a-d)\wedge b}}\left\{\begin{array}{ll}
			\varepsilon^{\frac{b}{2}}T^{1-H(a+b+d)}	,&\text{if }H(a+d)<1-Hb;\\
			\varepsilon^{\frac{b}{2}}\ln(1+\gamma^H\varepsilon^{-\frac{1}{2H}}T)	,&\text{if }1-Hb=H(a+d)<1;\\
			\varepsilon^{-\frac{1}{2}(a+d-\frac{1}{H})}	,&\text{if }1-Hb<H(a+d)<1.
		\end{array}\right.
	\end{align*}
\end{lemma}
\begin{proof}
	When $H(a+d)<1-Hb$, 
	\begin{align*}
		\varepsilon^{-\frac{b}{2}}\int_{0}^{T}\int_{\R^d}|x|^a\Big(\big|\frac{1}{\gamma}\varepsilon^{\frac{1}{2}}x\big|^b\wedge1\Big)e^{-|x|^2t^{2H}}dxdt&\le\frac{1}{\gamma^b}\int_{0}^{T}\int_{\R^d}|x|^{a+b}e^{-|x|^2t^{2H}}dxdt\\&\le \frac{C_{1,10}}{\gamma^b}\int_{0}^{T}t^{-H(a+b+d)}dt=\frac{C_{1,10}}{\gamma^b}T^{1-H(a+b+d)}.
	\end{align*}
	When $1-Hb\le H(a+d)<1$, 
	\begin{align*}
		&\varepsilon^{-\frac{b}{2}}\int_{0}^{T}\int_{\R^d}|x|^a\Big(\big|\frac{1}{\gamma}\varepsilon^{\frac{1}{2}}x\big|^b\wedge1\Big)e^{-|x|^2t^{2H}}dxdt\\=&\frac{\varepsilon^{-\frac{1}{2}(a+b+d-\frac{1}{H})}}{\gamma^{\frac{1}{H}-a-d}}\int_{0}^{\gamma^H\varepsilon^{-\frac{1}{2H}}T}\int_{\R^d}|x|^a\Big(|x|^b\wedge1\Big)e^{-|x|^2t^{2H}}dxdt\\\le&\frac{\varepsilon^{-\frac{1}{2}(a+b+d-\frac{1}{H})}}{\gamma^{\frac{1}{H}-a-d}}\Big(\int_{0}^{\gamma^H\varepsilon^{-\frac{1}{2H}}T}\int_{|x|\le1}|x|^{a+b}e^{-|x|^2t^{2H}}dxdt+\int_{0}^{\gamma^H\varepsilon^{-\frac{1}{2H}}T}\int_{|x|>1}|x|^ae^{-|x|^2t^{2H}}dxdt\Big)\\\le&\frac{\varepsilon^{-\frac{1}{2}(a+b+d-\frac{1}{H})}}{\gamma^{\frac{1}{H}-a-d}}C_{1,11}\Big(\int_{0}^{1}\int_{|x|\le1}|x|^{a+b}e^{-|x|^2t^{2H}}dxdt+\int_{1}^{\gamma^H\varepsilon^{-\frac{1}{2H}}T}\int_{|x|\le1}|x|^{a+b}e^{-|x|^2t^{2H}}dxdt+1\Big)\\\le&\frac{\varepsilon^{-\frac{1}{2}(a+b+d-\frac{1}{H})}}{\gamma^{\frac{1}{H}-a-d}}C_{1,12}\Big(\int_{1}^{\gamma^H\varepsilon^{-\frac{1}{2H}}T}t^{-H(a+b+d)}dt+1\Big).
	\end{align*}
	Then the desired results are easy to get.
\end{proof}
Finally we give a lemma for the process $\displaystyle\Big(W(L(T,x)),T\ge0\Big)$ in the limit of Theorem \ref{thmdis}, where $W$ is an standard Brownian motion independent of $X$.
\begin{lemma}\label{2lma4}
	Fix a finite number of disjoint intervals 
	$(a_i, b_i]$ in $[0,\infty)$, where $i = 1, \dots, n$  and $b_i \leq a_{i+1}$. 
	Consider a multi-index $\mathbf{m} = (m_1, \dots, m_n)$, where $m_i \geq 1$ and 
	$1 \leq  i \leq n$. Then 
	\[
	\E \Big(\prod_{i=1}^n \big[W (L(b_i,x) ) - W (L (a_i,x))\big]^{m_i} \Big)
	\]
	is equal to 
	\begin{equation}\label{moments}
		\bigg(\prod\limits_{i=1}^n \frac{m_i!}{ 2^{\frac{m_i}{2}}   (2\pi )^{\frac{m_id}{2} }   (\frac{m_i}{2})!} \bigg)\displaystyle \int_{\prod\limits_{i=1}^n [a_i ,b_i]^{\frac{m_i}{2}} }\displaystyle \int_{\mathbb{R}^{\frac{|\mathbf{m}|d}{2}}}  \exp\Big(\iota  x\cdot \sum\limits^n_{i=1} \sum\limits^{\frac{m_i}{2}}_{j=1}x^j_i-\frac{1}{2}\Var\big(\sum\limits^n_{i=1}\sum\limits^{\frac{m_i}{2}}_{j=1}x^j_i\cdot X_{u^j_i}\big)\Big)\, dx\, du 
	\end{equation}
	if all $m_i$ are even and $0$ otherwise, where $|\mathbf{m}|=\sum\limits^n_{i=1} m_i$. Moreover, the law of the random vector  $\big(W (L(b_i,x) ) - W (L (a_i,x)): 1\le i\le n\big)$ is determined by
	the moments in (\ref{moments}).
\end{lemma}
\begin{proof}
	We can get the result from Lemma 2.2 and  Theorem 6.1 in \cite{hlx2023}.
\end{proof}
\section{The proof of Theorem \ref{thmlp}}
 In this section, we will get the proof of Theorem \ref{thmlp}. Using Fourier inverse transform, we get 
\begin{align*}
	\Delta_{\varepsilon}(T)&:=\varepsilon^{-\frac{N}{2}}\bigg(\varepsilon^{-\frac{|\boldsymbol{k}|+d}{2}}\int_{0}^{T}f^{(\boldsymbol{k})}\big(\varepsilon^{-\frac12}(X_t+x)\big)dt-\widehat{f}(0)L^{(\boldsymbol{k})}(T,x)\bigg)\\&\qquad\qquad\qquad\qquad-\frac{\iota^{N}}{N!}\sum_{i_1,\cdots,i_{N}=1}^{d}L^{(\boldsymbol{k}+\boldsymbol{e}_{i_1}+\boldsymbol{e}_{i_2}+\cdots+\boldsymbol{e}_{i_{N}})}(T,x)\int_{\R^d}v_{i_1}v_{i_2}\cdots v_{i_{N}}f(v)dv\\&=\frac{1}{(2\pi)^d}\int_{0}^{T}\int_{\R^d}\Big(\varepsilon^{-\frac{N}{2}}\big(\widehat{f}(\varepsilon^{\frac{1}{2}}u)-\widehat{f}(0)\big)-\int_{\R^d}f(v)\frac{(\iota u\cdot v)^N}{N!}dv\Big)\prod_{\ell=1}^{d}\big(\iota u_{\ell}\big)^{k_\ell}e^{\iota(X_t+x)\cdot u}dudt.
\end{align*}
and for any $m\ge1$, 
\begin{align*}
	\Big|\E\big[\Delta_{\varepsilon}(T)\big]^{m}\Big|\le\frac{1}{(2\pi)^{md}}\int_{[0,T]^m}\int_{\R^{md}}\prod_{j=1}^{m}\Big|\varepsilon^{-\frac{N}{2}}&\big(\widehat{f}(\varepsilon^{\frac{1}{2}}u_j)-\widehat{f}(0)\big)-\int_{\R^d}f(v_j)\frac{(\iota u_j\cdot v_j)^N}{N!}dv_j\Big|\\&\times\prod_{j=1}^{m}|u_j|^{|\boldsymbol{k}|}\times\exp\Big(-\frac{1}{2}\Var\Big(\sum_{j=1}^{m}u_j\cdot X_{t_j}\Big)\Big)dudt.
\end{align*}
Note that  $\display\Big|\varepsilon^{-\frac{N}{2}}\big(\widehat{f}(\varepsilon^{\frac{1}{2}}u)-\widehat{f}(0)\big)-\int_{\R^d}f(v)\frac{(\iota u\cdot v)^N}{N!}dv\Big|
\le C_{2,1}|u|^N$
from Lemma \ref{2lma1} and 
\begin{align*}
	&\lim\limits_{\varepsilon\downarrow0}\varepsilon^{-\frac{N}{2}}\big(\widehat{f}(\varepsilon^{\frac{1}{2}}u)-\widehat{f}(0)\big)\\=&\lim\limits_{\varepsilon\downarrow0}\int_{\R^d}f(v)\varepsilon^{-\frac{N}{2}}\Big(e^{\iota\varepsilon^{\frac{1}{2}}(u\cdot v)}-1\Big)dv
	\\=&\lim\limits_{\varepsilon\downarrow0}\int_{\R^d}f(v)\varepsilon^{-\frac{N}{2}}\Big(e^{\iota\varepsilon^{\frac{1}{2}}(u\cdot v)}-1-\iota\varepsilon^{\frac{1}{2}}(u\cdot v)-\frac{(\iota\varepsilon^{\frac{1}{2}}(u\cdot v))^2}{2!}-\cdots-\frac{(\iota\varepsilon^{\frac{1}{2}}(u\cdot v))^{N-1}}{(N-1)!}\Big)dv
	\\
	=&\int_{\R^d}f(v)\frac{(\iota u\cdot v)^N}{N!}dv.
\end{align*}
To get the desired result, using dominated convergence theorem, we only need to prove 
\begin{align*}
	J:=\int_{[0,T]^m}\int_{\R^{md}}\prod_{j=1}^{m}|u_j|^{|\boldsymbol{k}|+N}\times\exp\Big(-\frac{1}{2}\Var\Big(\sum_{j=1}^{m}u_j\cdot X_{t_j}\Big)\Big)dudt<+\infty
\end{align*}
when $H(2|\boldsymbol{k}|+d)<1-2NH$, which comes from 
\begin{align*}
	J&=m!\int_{[0,T]^m_{<}}\int_{\R^{md}}\prod_{j=1}^{m}|v_j-v_{j+1}|^{|\boldsymbol{k}|+N}\times\exp\Big(-\frac{1}{2}\Var\Big(\sum_{j=1}^{m}v_j\cdot \big(X_{t_j}-X_{t_{j-1}}\big)\Big)\Big)dudt\\&\le C_{2,2}\int_{[0,T]^m_{<}}\int_{\R^{md}}\prod_{j=1}^{m}\Big(|v_j|^{|\boldsymbol{k}|+N}+|v_{j+1}|^{|\boldsymbol{k}|+N}\Big)\times\exp\Big(-\frac{\kappa_{H,m}}{2}\sum_{j=1}^{m}|v_j|^2(t_j-t_{j-1})^{2H}\Big)dudt\\&\le C_{2,3}\int_{[0,T]^m}\int_{\R^{md}}\sum_{\mS_{m}}\Big(\prod_{j=1}^{m}|v_j|^{(|\boldsymbol{k}|+N)(\bar{p}_{j-1}+p_{j})}\Big)\times\exp\Big(-\frac{\kappa_{H,m}}{2}\sum_{j=1}^{m}|v_j|^2s_j^{2H}\Big)duds\\&\le C_{2,4}\sum_{\mS_{m}}\prod_{j=1}^{m}\int_{0}^{T}s_j^{-H(|\boldsymbol{k}|+N)(\bar{p}_{j-1}+p_{j})-Hd}ds_j<\infty
,\end{align*}
where $[0,T]^m_<=\{(t_1,t_2,\cdots,t_m):0=t_0<t_1<t_2<\cdots<t_m<T\}$, in the first equality we use notation $v_{m+1}=0$ and let $v_j-v_{j+1}=u_j,j=1,2,\cdots,m$, in the first inequality we use the local non-determinism (\ref{local}), in the second inequality we let $t_j-t_{j-1}=s_j,j=1,2,\cdots,m$ and 
\begin{align}\label{mS}
	\mS_m=\Big\{p_j,\bar{p}_j:p_j\in\{0,1\},p_j+\bar{p}_j=1,j=1,2,\cdots,m-1,\bar{p}_0=0,p_{m}=1\Big\}.
\end{align}
\section{The proof of Theorem \ref{thmdis}}
In this section we will give the proof of Theorem \ref{thmdis}. Denote 
\begin{align*}
	F_{\varepsilon}(T)&=\ell_{\varepsilon,H,d}^{(|\boldsymbol{k}|,N)}\bigg(\varepsilon^{-\frac{|\boldsymbol{k}|+d}{2}}\int_{0}^{T}f^{(\boldsymbol{k})}\big(\varepsilon^{-\frac12}(X_t+x)\big)dt-\widehat{f}(0)L^{(\boldsymbol{k})}(T,x)\bigg).
\end{align*}
Using Fourier transform, we get 
\begin{align*}
	F_{\varepsilon}(T)=\ell_{\varepsilon,H,d}^{(|\boldsymbol{k}|,N)}\bigg(\frac{1}{(2\pi)^d}\int_{0}^{T}\int_{\mathbb{R}^d}\prod_{\ell=1}^{d}\big(\iota u_{\ell})^{k_\ell}\Big(\widehat{f}(\varepsilon^{\frac12}u)-\widehat{f}(0)\Big)\exp\big(\iota(X_t+x)\cdot u\big)dudt\bigg).
\end{align*}
To get Theorem \ref{thmdis}, we should  confirm the tightness of processes $\Big(F_{\varepsilon}(T):T>0\Big)$, which comes from the following proposition:
\begin{proposition}\label{tightness}
	Given any integer $m\ge1$, there exists a positive constant $C$ depending on $T,\boldsymbol{k},f,N,H,d$, s.t. 
	\begin{align*}
		\Big|\E\big[F_{\varepsilon}(T_2)-F_{\varepsilon}(T_1)\big]^{m}\Big|\le C\big(T_2-T_1\big)^{\frac{m}{2}(1-Hd)}
	\end{align*}
	for any $0\le T_1<T_2\le T$.
\end{proposition}
\begin{proof}
	By the definition of $F_{\varepsilon}(T)$, 
	\begin{align*}
		\E\big[F_{\varepsilon}(T_2)-F_{\varepsilon}(T_1)\big]^{m}=\bigg(\ell_{\varepsilon,H,d}^{(|\boldsymbol{k}|,N)}\bigg)^m&\frac{1}{(2\pi)^{md}}\int_{[T_1,T_2]^m}\int_{\R^{md}}\prod_{j=1}^{m}\Big(\prod_{\ell=1}^{d}\big(\iota u_{j\ell})^{k_\ell}\Big)\\\times\prod_{j=1}^{m}&\Big(\widehat{f}(\varepsilon^{\frac12}u_j)-\widehat{f}(0)\Big)\exp\Big(-\frac{1}{2}\Var\Big(\sum_{j=1}^{m}u_j\cdot X_{t_j}\Big)+\iota\sum_{j=1}^{m}u_j\cdot x\Big)dudt\\=\bigg(\ell_{\varepsilon,H,d}^{(|\boldsymbol{k}|,N)}\bigg)^m&\frac{m!}{(2\pi)^{md}}\int_{[T_1,T_2]^m_<}\int_{\R^{md}}\prod_{j=1}^{m}\Big(\prod_{\ell=1}^{d}\big(\iota u_{j\ell})^{k_\ell}\Big)\\\times\prod_{j=1}^{m}&\Big(\widehat{f}(\varepsilon^{\frac12}u_j)-\widehat{f}(0)\Big)\exp\Big(-\frac{1}{2}\Var\Big(\sum_{j=1}^{m}u_j\cdot X_{t_j}\Big)+\iota\sum_{j=1}^{m}u_j\cdot x\Big)dudt,
	\end{align*}
	where $\display[T_1,T_2]^m_<=\big\{(t_1,t_2,\cdots,t_m):T_1<t_1<t_2<\cdots<t_m<T_2\big\}$. Then using coordinate transform $v_j=u_j-u_{j+1},j=1,2,\cdots,m$ with $v_{m+1}=0$, we get 
	\begin{align*}
		&\Big|\E\big[F_{\varepsilon}(T_2)-F_{\varepsilon}(T_1)\big]^{m}\Big|\\\le& C^{(1)}_{4,1}\bigg(\ell_{\varepsilon,H,d}^{(|\boldsymbol{k}|,N)}\bigg)^m\int_{[T_1,T_2]^m_<}\int_{\R^{md}}\prod_{j=1}^{m}\Big|\prod_{\ell=1}^{d}\big(\iota v_{j\ell}-\iota v_{j+1,\ell}\big)^{k_\ell}\Big|\\&\qquad\qquad\qquad\times\prod_{j=1}^{m}\Big|\widehat{f}\big(\varepsilon^{\frac12}(v_j-v_{j+1})\big)-\widehat{f}(0)\Big|\exp\Big(-\frac{1}{2}\Var\Big(\sum_{j=1}^{m}v_j\cdot \big(X_{t_j}-X_{t_{j-1}}\big)\Big)\Big)dvdt\\\le& C^{(1)}_{4,2}\bigg(\ell_{\varepsilon,H,d}^{(|\boldsymbol{k}|,N)}\bigg)^m\int_{[T_1,T_2]^m_<}\int_{\R^{md}}\prod_{j=1}^{m}\Big(|v_{j}|^{|\boldsymbol{k}|}\big(|\varepsilon^{\frac12}v_j|\wedge1\big)^N+|v_{j+1}|^{|\boldsymbol{k}|}\big(|\varepsilon^{\frac12}v_{j+1}|\wedge1\big)^N\Big)\\&\qquad\qquad\qquad\qquad\qquad\qquad\qquad\qquad\qquad\qquad\times\exp\Big(-\frac{\kappa_{H,m}}{2}\sum_{j=1}^{m}|v_j|^2\big({t_j}-{t_{j-1}}\big)^{2H}\Big)dvdt,
	\end{align*}
	where in the second inequality we use Lemma \ref{2lma1} to get $$\big|\widehat{f}\big(\varepsilon^{\frac12}(v_j-v_{j+1})\big)-\widehat{f}(0)\big|\le C^{(1)}_{4,3}\big(|v_j-v_{j+1}|\wedge1\big)^N\le C^{(1)}_{4,4}\Big(\big(|v_j|\wedge1\big)^N+\big(|v_{j+1}|\wedge1\big)^N\Big).$$
	Then using coordinate transform $s_j=t_{j}-t_{j-1},j=1,2,\cdots,m$ with $t_0=T_1$, we get 
	\begin{align*}
		\Big|\E&\big[F_{\varepsilon}(T_2)-F_{\varepsilon}(T_1)\big]^{m}\Big|\le\\& C^{(1)}_{4,5}\bigg(\ell_{\varepsilon,H,d}^{(|\boldsymbol{k}|,N)}\bigg)^m\sum_{\mS_{m}}\int_{[0,T_2-T_1]^m}\int_{\R^{md}}\prod_{j=1}^{m}\Big(|v_{j}|^{|\boldsymbol{k}|}\big(|\varepsilon^{\frac12}v_j|\wedge1\big)^N\Big)^{\bar{p}_{j-1}+p_j}\exp\Big(-\frac{\kappa_{H,m}}{2}\sum_{j=1}^{m}|v_j|^2s_j^{2H}\Big)dvds,
	\end{align*}
	where $\mS_{m}$ is defined in (\ref{mS}). Note that $\bar{p}_{j-1}+p_j\in\{0,1,2\}$ and $\display\sum_{j=1}^{m}(\bar{p}_{j-1}+p_j)=m$, using Lemma \ref{5lma3} in the case $a=|\boldsymbol{k}|,b=N$, there exists a constant $C^{(1)}_{4,6}>0$, s.t. 
	\begin{align*}
		\Big|\E\big[F_{\varepsilon}(T_2)-F_{\varepsilon}(T_1)\big]^{m}\Big|\le C^{(1)}_{4,6}\big(T_2-T_1\big)^{\frac{m}{2}(1-Hd)}
	\end{align*}
	for any $0\le T_1<T_2\le T$.
\end{proof}
Then we only need to prove that the processes $\Big(F_{\varepsilon}(T):T>0\Big)$ converges to $\displaystyle\Big(W(L(T,x)),T\ge0\Big)$ in finite dimensional distribution. From Lemma \ref{2lma4}, we shall prove that fix a finite number of disjoint intervals 
$(a_i, b_i]$ in $[0,\infty)$, $i = 1, \dots, n$  with $b_i \leq a_{i+1}$and a multi-index $\mathbf{m} = (m_1, \dots, m_n)$, with $m_i \geq 1$, $1 \leq  i \leq n$, there exists
\begin{align}\label{momenttarget}
	\lim\limits_{\varepsilon\downarrow0}\E \Big(\prod_{i=1}^n \big[F_{\varepsilon}(b_i) - F_{\varepsilon}(a_i)\big]^{m_i} \Big)=\E \Big(\prod_{i=1}^n \big[W (L(b_i,x) ) - W (L (a_i,x))\big]^{m_i} \Big).
\end{align}
 We divide the proof of (\ref{momenttarget}) into three parts: (i) $|\boldsymbol{m}|$ is odd (ii) $|\boldsymbol{m}|$ is even but one of $m_i$ is odd; (ii) all $m_i$ are even. Note that 
 \begin{align*}
 	\E \Big(\prod_{i=1}^n \big[F_{\varepsilon}(b_i) &- F_{\varepsilon}(a_i)\big]^{m_i} \Big)=\bigg(\ell_{\varepsilon,H,d}^{(|\boldsymbol{k}|,N)}\bigg)^{|\boldsymbol{m}|}\frac{1}{(2\pi)^{|\boldsymbol{m}|d}}\int_{\prod\limits_{i=1}^{n}[a_i,b_i]^{m_i}}\int_{\R^{|\boldsymbol{m}|d}}\Big( \prod^{n}_{i=1}\prod^{m_i}_{j=1}\prod_{\ell=1}^d\big(\iota u^i_{j\ell}\big)^{k_\ell}\Big) \\&\qquad\qquad\times\Big( \prod^{n}_{i=1}\prod^{m_i}_{j=1} (\widehat{f}(\varepsilon^{\frac12}u^i_j)-\widehat{f}(0))\Big)\exp\Big(\iota x\cdot \sum_{i=1}^{n}\sum_{j=1}^{m_i}u^i_j-\frac{1}{2}\Var\Big(\sum_{i=1}^{n}\sum_{j=1}^{m_i}u^i_j\cdot X_{t^i_j}\Big)\Big)dudt\\&\qquad\qquad\qquad=\bigg(\ell_{\varepsilon,H,d}^{(|\boldsymbol{k}|,N)}\bigg)^{|\boldsymbol{m}|}\frac{\boldsymbol{m}!}{(2\pi)^{|\boldsymbol{m}|d}}\int_{D_{\boldsymbol{m}}}\int_{\R^{|\boldsymbol{m}|d}}\Big( \prod^{n}_{i=1}\prod^{m_i}_{j=1}\prod_{\ell=1}^d\big(\iota u^i_{j\ell}\big)^{k_\ell}\Big) \\&\qquad\qquad\qquad\times\Big( \prod^{n}_{i=1}\prod^{m_i}_{j=1} (\widehat{f}(\varepsilon^{\frac12}u^i_j)-\widehat{f}(0))\Big)\exp\Big(\iota x\cdot \sum_{i=1}^{n}\sum_{j=1}^{m_i}u^i_j-\frac{1}{2}\Var\Big(\sum_{i=1}^{n}\sum_{j=1}^{m_i}u^i_j\cdot X_{t^i_j}\Big)\Big)dudt,
 \end{align*}
 where $u_{j}^i=(u_{j1}^i,\cdots,u_{jd}^i)$ for $i,j$, $\sum\limits_{i=1}^n m_i=|\mathbf{m}|$,  $\prod\limits_{i=1}^n m_i!=\mathbf{m}!$ and \begin{equation}\label{Dm}
 	D_{\boldsymbol{m}}=\big\{t \in \R^{|\boldsymbol{m}|}: a_{i}<t^i_{1}<\cdots <t^i_{m_i}<b_{i}, 1\le i\le n\big\}.
 \end{equation}
Define the index set of $u_j^i$'s above as 
\[
J_0=\big\{(i,j): 1\leq i\leq N, 1\leq j\leq m_i \big\},
\]
where for any $(i_1, j_1)$ and $(i_2,j_2)\in J_0$,  the following dictionary ordering 
\[
(i_1, j_1)\leq (i_2,j_2)
\]
exists if $i_1<i_2$ or $i_1=i_2$ and $j_1\leq j_2$. For any $(i,j)$ in $J_0$, let $\#(i,j)=\sum\limits^{i-1}_{k=1}m_k+j$, which means that $(i,j)$ is the $(\sum\limits^{i-1}_{k=1}m_k+j)$-th element in $J_0$ under the above ordering. Using these notations, 
\begin{align*}
	\E \Big(\prod_{i=1}^n \big[F_{\varepsilon}(b_i) - F_{\varepsilon}(a_i)\big]^{m_i} \Big)=&\bigg(\ell_{\varepsilon,H,d}^{(|\boldsymbol{k}|,N)}\bigg)^{|\boldsymbol{m}|}\frac{\boldsymbol{m}!}{(2\pi)^{|\boldsymbol{m}|d}}\int_{D_{\boldsymbol{m}}}\int_{\R^{|\boldsymbol{m}|d}}\Big( \prod^{|\boldsymbol{m}|}_{j=1}\prod_{\ell=1}^d\big(\iota u_{j\ell}\big)^{k_\ell}\Big) \\&\times\Big( \prod^{|\boldsymbol{m}|}_{j=1} (\widehat{f}(\varepsilon^{\frac{1}{2}}u_j)-\widehat{f}(0))\Big)\exp\Big(\iota x\cdot \sum_{j=1}^{|\boldsymbol{m}|}u_j-\frac{1}{2}\Var\Big(\sum^{|\boldsymbol{m}|}_{j=1}u_j\cdot X_{t_j}\Big)\Big)dudt.
\end{align*}
Then recalling Lemma \ref{2lma4}, we only need to prove the following propositions:
\begin{proposition}\label{odd1}
	Assume $|\boldsymbol{m}|$ is odd. We have 
	\begin{align*}
		\lim\limits_{\varepsilon\downarrow0}\E \Big(\prod_{i=1}^n \big[F_{\varepsilon}(b_i) - &F_{\varepsilon}(a_i)\big]^{m_i} \Big)=\lim\limits_{\varepsilon\downarrow0}\bigg(\ell_{\varepsilon,H,d}^{(|\boldsymbol{k}|,N)}\bigg)^{|\boldsymbol{m}|}\frac{\boldsymbol{m}!}{(2\pi)^{|\boldsymbol{m}|d}}\int_{D_{\boldsymbol{m}}}\int_{\R^{|\boldsymbol{m}|d}}\Big( \prod^{|\boldsymbol{m}|}_{j=1}\prod_{\ell=1}^d\big(\iota u_{j\ell}\big)^{k_\ell}\Big) \\&\times\Big( \prod^{|\boldsymbol{m}|}_{j=1} (\widehat{f}(\varepsilon^{\frac{1}{2}}u_j)-\widehat{f}(0))\Big)\exp\Big(\iota x\cdot \sum_{j=1}^{|\boldsymbol{m}|}u_j-\frac{1}{2}\Var\Big(\sum^{|\boldsymbol{m}|}_{j=1}u_j\cdot X_{t_j}\Big)\Big)dudt=0
	\end{align*}
\end{proposition}
\begin{proposition}\label{odd2}
	Assume $|\boldsymbol{m}|$ is even but some $m_{i}$ is odd. We have 
		\begin{align*}
		\lim\limits_{\varepsilon\downarrow0}\E \Big(\prod_{i=1}^n \big[F_{\varepsilon}(b_i) -& F_{\varepsilon}(a_i)\big]^{m_i} \Big)=\lim\limits_{\varepsilon\downarrow0}\bigg(\ell_{\varepsilon,H,d}^{(|\boldsymbol{k}|,N)}\bigg)^{|\boldsymbol{m}|}\frac{\boldsymbol{m}!}{(2\pi)^{|\boldsymbol{m}|d}}\int_{D_{\boldsymbol{m}}}\int_{\R^{|\boldsymbol{m}|d}}\Big( \prod^{|\boldsymbol{m}|}_{j=1}\prod_{\ell=1}^d\big(\iota u_{j\ell}\big)^{k_\ell}\Big) \\&\times\Big( \prod^{|\boldsymbol{m}|}_{j=1} (\widehat{f}(\varepsilon^{\frac{1}{2}}u_j)-\widehat{f}(0))\Big)\exp\Big(\iota x\cdot \sum_{j=1}^{|\boldsymbol{m}|}u_j-\frac{1}{2}\Var\Big(\sum^{|\boldsymbol{m}|}_{j=1}u_j\cdot X_{t_j}\Big)\Big)dudt=0
	\end{align*}
\end{proposition}

\begin{proposition}\label{alleven}
Assume all 	$m_{i_0}$ is even. We have 
	\begin{align*}
&\lim\limits_{\varepsilon\downarrow0}\E \Big(\prod_{i=1}^n \big[F_{\varepsilon}(b_i) - F_{\varepsilon}(a_i)\big]^{m_i} \Big)=	\lim\limits_{\varepsilon\downarrow0}\bigg(\ell_{\varepsilon,H,d}^{(|\boldsymbol{k}|,N)}\bigg)^{|\boldsymbol{m}|}\frac{\boldsymbol{m}!}{(2\pi)^{|\boldsymbol{m}|d}}\int_{D_{\boldsymbol{m}}}\int_{\R^{|\boldsymbol{m}|d}}\Big( \prod^{|\boldsymbol{m}|}_{j=1}\prod_{\ell=1}^d\big(\iota u_{j\ell}\big)^{k_\ell}\Big) \\&\quad\quad\qquad\qquad\qquad\qquad\times\Big( \prod^{|\boldsymbol{m}|}_{j=1} (\widehat{f}(\varepsilon^{\frac{1}{2}}u_j)-\widehat{f}(0))\Big)\exp\Big(\iota x\cdot \sum_{j=1}^{|\boldsymbol{m}|}u_j-\frac{1}{2}\Var\Big(\sum^{|\boldsymbol{m}|}_{j=1}u_j\cdot X_{t_j}\Big)\Big)dudt\\=&\Big(D_{H,d}\Big)^{\frac{|\boldsymbol{m}|}{2}}\bigg(\frac{\boldsymbol{m}!}{ 2^{\frac{|\boldsymbol{m}|}{2}}   (2\pi )^{\frac{|\boldsymbol{m}|d}{2} }   (\frac{\boldsymbol{m}}{2})!} \bigg)\displaystyle \int_{\prod\limits_{i=1}^n [a_i ,b_i]^{\frac{m_i}{2}} }\displaystyle \int_{\mathbb{R}^{\frac{|\mathbf{m}|d}{2}}}  \exp\Big(\iota  x\cdot \sum\limits^{\frac{|\boldsymbol{m}|}{2}}_{j=1}u_j-\frac{1}{2}\Var\big(\sum\limits^{\frac{|\boldsymbol{m}|}{2}}_{j=1}u_j\cdot X_{t_j}\big)\Big)dudt\\=&\Big(D_{H,d}\Big)^{\frac{|\boldsymbol{m}|}{2}}\E \Big(\prod_{i=1}^n \big[W (L(b_i,x) ) - W (L (a_i,x))\big]^{m_i} \Big),
\end{align*}
	where $D_{H,d}$ is defined in (\ref{Dhd}).
\end{proposition}
In the rest of this section we will give the proof of these propositions.
\begin{proof}[\bf The proof of Proposition \ref{odd1}:]
Let $T=b_{b_n}$. Like the process in the proof of Proposition \ref{tightness}, we have 
\begin{align*}
	&\Big|\E \Big(\prod_{i=1}^n \big[F_{\varepsilon}(b_i) - F_{\varepsilon}(a_i)\big]^{m_i} \Big)\Big|\\\le& C^{(2)}_{4,1}\bigg(\ell_{\varepsilon,H,d}^{(|\boldsymbol{k}|,N)}\bigg)^{|\boldsymbol{m}|}\int_{[0,T]^{|\boldsymbol{m}|}_<}\int_{\R^{|\boldsymbol{m}|d}}\prod_{j=1}^{|\boldsymbol{m}|}\Big(|v_{j}|^{|\boldsymbol{k}|}\big(|\varepsilon^{\frac12}v_j|\wedge1\big)^N+|v_{j+1}|^{|\boldsymbol{k}|}\big(|\varepsilon^{\frac12}v_{j+1}|\wedge1\big)^N\Big)\\&\qquad\qquad\qquad\qquad\qquad\qquad\qquad\qquad\qquad\qquad\qquad\times\exp\Big(-\frac{\kappa_{H,|\boldsymbol{m}|}}{2}\sum_{j=1}^{|\boldsymbol{m}|}|v_j|^2\big({t_j}-{t_{j-1}}\big)^{2H}\Big)dvdt\\\le&C^{(2)}_{4,2}\bigg(\ell_{\varepsilon,H,d}^{(|\boldsymbol{k}|,N)}\bigg)^{|\boldsymbol{m}|}\int_{[0,T]^{|\boldsymbol{m}|}}\int_{\R^{|\boldsymbol{m}|d}}\sum_{\mS_{|\boldsymbol{m}|}}\prod_{j=1}^{|\boldsymbol{m}|}\Big(|v_{j}|^{|\boldsymbol{k}|}\big(|\varepsilon^{\frac12}v_j|\wedge1\big)^N\Big)^{\bar{p}_{j-1}+p_j}\exp\Big(-\frac{\kappa_{H,|\boldsymbol{m}|}}{2}\sum_{j=1}^{|\boldsymbol{m}|}|v_j|^2s_j^{2H}\Big)dvds,
\end{align*}
where is $\mS_{|\boldsymbol{m}|}$ defined like (\ref{mS}). Note that $\bar{p}_{j-1}+p_j\in\{0,1,2\}$ and $\display\sum_{j=1}^{|\boldsymbol{m}|}(\bar{p}_{j-1}+p_j)=|\boldsymbol{m}|$, using Lemma \ref{5lma3} in the case $a=|\boldsymbol{k}|,b=N$, we get for some $\delta>0$,
\begin{align*}
	\Big|\E \Big(\prod_{i=1}^n \big[F_{\varepsilon}(b_i) - F_{\varepsilon}(a_i)\big]^{m_i} \Big)\Big|\le C_{4,3}^{(2)}\begin{cases}
		\ln^{-\frac{1}{2}}(1+\varepsilon^{-\frac12}),&\textit{if }1-2NH=H(2|\boldsymbol{k}|+d);\\\varepsilon^{-\delta},&\textit{if }1-2NH<H(2|\boldsymbol{k}|+d)<1,
	\end{cases}
\end{align*}
because $|\boldsymbol{m}|$ is odd so that the number of $1$'s in $\{\bar{p}_{j-1}+p_j:j=1,2,\cdots,|\boldsymbol{m}|\}$ is greater than or equal to one, which completes the proof. 	
\end{proof}
\begin{proof}[\bf The proof of Proposition \ref{odd2}:]
Denote $i_0$ the smallest $i$ s.t. $m_i$ is odd. Then let $n_1=\sum\limits_{i=1}^{i_0}m_i$, $n_2=\sum\limits_{i=i_0+1}^{|\boldsymbol{m}|}m_i$, $T_1=b_{i_0}$, $T_2=a_{i_0+1}$, $T=b_{n}$
and $$D_{12}=\big\{t\in\R^{|\boldsymbol{m}|}:0<t_1<\cdots<t_{n_1}<T_1,T_2<t_{n_1+1}<\cdots<t_{|\boldsymbol{m}|}<T\}.$$
Recall Proposition \ref{tightness}, we only need to consider the case $b_{i_0}<a_{i_0+1}$ in the proof of Proposition \ref{odd2}. Like the proof of Proposition \ref{odd1}, we have
\begin{align*}
	&\Big|\E \Big(\prod_{i=1}^n \big[F_{\varepsilon}(b_i) - F_{\varepsilon}(a_i)\big]^{m_i} \Big)\Big|\\\le& C^{(3)}_{4,1}\bigg(\ell_{\varepsilon,H,d}^{(|\boldsymbol{k}|,N)}\bigg)^{|\boldsymbol{m}|}\int_{D_{12}}\int_{\R^{|\boldsymbol{m}|d}}\prod_{j=1}^{|\boldsymbol{m}|}\Big(|v_{j}|^{|\boldsymbol{k}|}\big(|\varepsilon^{\frac12}v_j|\wedge1\big)^N+|v_{j+1}|^{|\boldsymbol{k}|}\big(|\varepsilon^{\frac12}v_{j+1}|\wedge1\big)^N\Big)\\&\qquad\qquad\qquad\qquad\qquad\qquad\qquad\qquad\qquad\qquad\qquad\times\exp\Big(-\frac{\kappa_{H,|\boldsymbol{m}|}}{2}\sum_{j=1}^{|\boldsymbol{m}|}|v_j|^2\big({t_j}-{t_{j-1}}\big)^{2H}\Big)dvdt\\\le&C^{(3)}_{4,2}\bigg(\ell_{\varepsilon,H,d}^{(|\boldsymbol{k}|,N)}\bigg)^{|\boldsymbol{m}|}\int_{D_{12}}\int_{\R^{|\boldsymbol{m}|d}}\sum_{\mS_{|\boldsymbol{m}|}}\prod_{j=1}^{|\boldsymbol{m}|}\Big(|v_{j}|^{|\boldsymbol{k}|}\big(|\varepsilon^{\frac12}v_j|\wedge1\big)^N\Big)^{\bar{p}_{j-1}+p_j}\\&\qquad\qquad\qquad\qquad\qquad\qquad\qquad\qquad\qquad\qquad\qquad\times\exp\Big(-\frac{\kappa_{H,|\boldsymbol{m}|}}{2}\sum_{j=1}^{|\boldsymbol{m}|}|v_j|^2\big({t_j}-{t_{j-1}}\big)^{2H}\Big)dvdt.
\end{align*}
Because $|\boldsymbol{m}|$ is even, $\mS_{|\boldsymbol{m}|}$ can be divided into two disjoint parts
\begin{align*}
	\mS_{|\boldsymbol{m}|,0}=\mS_{|\boldsymbol{m}|}\cap\Big\{	p_j,\bar{p}_j:\bar{p}_j=0\textit{ when j even},\bar{p}_j=1\textit{ when j odd},p_j+\bar{p}_j=1,j=1,2,\cdots,m-1\Big\}
\end{align*}
and 
\begin{align*}
	\mS_{|\boldsymbol{m}|,1}=\mS_{|\boldsymbol{m}|}-\mS_{|\boldsymbol{m}|,0}.
\end{align*}
where the number of $1$'s in $\{\bar{p}_{j-1}+p_j:j=1,2,\cdots,|\boldsymbol{m}|\}$ is greater than or equal to one for $p_j,\bar{p}_j$'s in $\mS_{|\boldsymbol{m}|,1}$ and the number of $1$'s in $\{\bar{p}_{j-1}+p_j:j=1,2,\cdots,|\boldsymbol{m}|\}$ is zero for $p_j,\bar{p}_j$'s in $\mS_{|\boldsymbol{m}|,0}$. So that 
\begin{align}
	&\Big|\E \Big(\prod_{i=1}^n \big[F_{\varepsilon}(b_i) - F_{\varepsilon}(a_i)\big]^{m_i} \Big)\Big|\nonumber\\\le& C^{(3)}_{4,3}\bigg(\ell_{\varepsilon,H,d}^{(|\boldsymbol{k}|,N)}\bigg)^{|\boldsymbol{m}|}\int_{D_{12}}\int_{\R^{|\boldsymbol{m}|d}}\prod_{j=1}^{|\boldsymbol{m}|/2}\Big(|v_{2j}|^{|\boldsymbol{k}|}\big(|\varepsilon^{\frac12}v_{2j}|\wedge1\big)^N\Big)^2\exp\Big(-\frac{\kappa_{H,|\boldsymbol{m}|}}{2}\sum_{j=1}^{|\boldsymbol{m}|}|v_j|^2\big({t_j}-{t_{j-1}}\big)^{2H}\Big)dvdt\nonumber\\&\qquad\qquad\qquad\qquad\qquad\qquad+C^{(3)}_{4,3}\begin{cases}
		\ln^{-\frac{1}{2}}(1+\varepsilon^{-\frac12}),&\textit{if }1-2NH=H(2|\boldsymbol{k}|+d);\\\varepsilon^{-\delta},&\textit{if }1-2NH<H(2|\boldsymbol{k}|+d)<1,\label{division}
	\end{cases}
\end{align}
using Lemma \ref{5lma3} for those $p_j,\bar{p}_j$'s in $\mS_{|\boldsymbol{m}|,1}$. Note that 
\begin{align*}
	&\sup_{t_{n_1-1}\in(0,T_1),\atop t_{n_1+2}\in(T_2,T)}\int_{t_{n_1-1}}^{T_1}\int_{T_2}^{t_{n_1+2}}\int_{\R^{2d}}\Big(|v_{n_1+1}|^{|\boldsymbol{k}|}\big(|\varepsilon^{\frac12}v_{n_1+1}|\wedge1\big)^N\Big)^2\\&\qquad\qquad\qquad\qquad\qquad\qquad\qquad\times\exp\Big(-\frac{\kappa_{H,|\boldsymbol{m}|}}{2}\sum_{j=n_1}^{n_1+1}|v_j|^2\big({t_j}-{t_{j-1}}\big)^{2H}\Big)dv_{n_1+1}dv_{n_1}dt_{n_1}dt_{n_1+1}\\&\le \varepsilon^{N}C^{(3)}_{4,3}\sup_{t_{n_1-1}\in(0,T_1)}\int_{t_{n_1-1}}^{T_1}\int_{T_2}^{T}\int_{\R^{d}}|v_{n_1+1}|^{2|\boldsymbol{k}|+2N}\big(t_{n_1}-t_{n_1-1}\big)^{-Hd}\\&\qquad\qquad\qquad\qquad\qquad\qquad\qquad\times\exp\Big(-\frac{\kappa_{H,|\boldsymbol{m}|}}{2}|v_{n_1+1}|^2\big(T_2-T_1\big)^{2H}\Big)dv_{n_1+1}dt_{n_1}dt_{n_1+1}\\&\le\varepsilon^{N}C^{(3)}_{4,4}(T-T_2)\big(T_2-T_1\big)^{-H(2|\boldsymbol{k}|+2N+d)}\sup_{t_{n_1-1}\in(0,T_1)}\int_{t_{n_1-1}}^{T_1}\big(t_{n_1}-t_{n_1-1}\big)^{-Hd}dt_{n_1}\le \varepsilon^{N}C^{(3)}_{4,5},
\end{align*}
we get 
\begin{align*}
	&\bigg(\ell_{\varepsilon,H,d}^{(|\boldsymbol{k}|,N)}\bigg)^{|\boldsymbol{m}|}\int_{D_{12}}\int_{\R^{|\boldsymbol{m}|d}}\prod_{j=1}^{|\boldsymbol{m}|/2}\Big(|v_{2j}|^{|\boldsymbol{k}|}\big(|\varepsilon^{\frac12}v_{2j}|\wedge1\big)^N\Big)^2\exp\Big(-\frac{\kappa_{H,|\boldsymbol{m}|}}{2}\sum_{j=1}^{|\boldsymbol{m}|}|v_j|^2\big({t_j}-{t_{j-1}}\big)^{2H}\Big)dvdt\\\le&\varepsilon^{N}C^{(3)}_{4,6}\bigg(\ell_{\varepsilon,H,d}^{(|\boldsymbol{k}|,N)}\bigg)^{|\boldsymbol{m}|}\int_{[0,T]^{|\boldsymbol{m}|-2}}\int_{\R^{|\boldsymbol{m}|d}}\prod_{j=1,\atop j\ne\frac{n_1+1}{2}}^{|\boldsymbol{m}|/2}\Big(|v_{2j}|^{|\boldsymbol{k}|}\big(|\varepsilon^{\frac12}v_{2j}|\wedge1\big)^N\Big)^2\\&\qquad\qquad\qquad\qquad\qquad\qquad\qquad\qquad\qquad\times\exp\Big(-\frac{\kappa_{H,|\boldsymbol{m}|}}{2}\sum_{j=1,\atop j\ne n_1,n_1+1}^{|\boldsymbol{m}|}|v_j|^2\big({t_j}-{t_{j-1}}\big)^{2H}\Big)dvdt\\\le& \varepsilon^{N}C^{(3)}_{4,7}\bigg(\ell_{\varepsilon,H,d}^{(|\boldsymbol{k}|,N)}\bigg)^{2}=C^{(3)}_{4,7}\left\{\begin{array}{ll}
		\ln^{-1}(1+\varepsilon^{-\frac12}), &\textit{if } 1-2NH=H(2|\boldsymbol{k}|+d);\\
		\varepsilon^{\frac{1}{2}(2|\boldsymbol{k}|+2N+d-\frac{1}{H})}, &\textit{if } 1-2NH<H(2|\boldsymbol{k}|+d)<1;
	\end{array}\right.
\end{align*}
using Lemma \ref{2lma3}, which completes the proof of Proposition \ref{odd2}.	
\end{proof}

\begin{proof}[\bf The proof of Proposition \ref{alleven}:]
	Using coordinate transform $u_j=y_j-y_{j+1},j=1,2,\cdots,|\boldsymbol{m}|$, $y_{|\boldsymbol{m}|+1}=0$ let $t_0=0$, we get 
	\begin{align*}
		\E \Big(\prod_{i=1}^n \big[F_{\varepsilon}(b_i) -& F_{\varepsilon}(a_i)\big]^{m_i} \Big)=\bigg(\ell_{\varepsilon,H,d}^{(|\boldsymbol{k}|,N)}\bigg)^{|\boldsymbol{m}|}\frac{\boldsymbol{m}!}{(2\pi)^{|\boldsymbol{m}|d}}\int_{D_{\boldsymbol{m}}}\int_{\R^{|\boldsymbol{m}|d}}\Big( \prod^{|\boldsymbol{m}|}_{j=1}\prod_{\ell=1}^d\big(\iota y_{j\ell}-\iota y_{j+1,\ell}\big)^{k_\ell}\Big) \\&\times\Big( \prod^{|\boldsymbol{m}|}_{j=1} (\widehat{f}\big(\varepsilon^{\frac{1}{2}}(y_j-y_{j+1})\big)-\widehat{f}(0))\Big)\exp\Big(\iota x\cdot y_1-\frac{1}{2}\Var\Big(\sum^{|\boldsymbol{m}|}_{j=1}y_j\cdot \big(X_{t_j}-X_{t_{j-1}}\big)\Big)\Big)dydt.
	\end{align*}
	
	Let $t_0=t_{-1}=0$. Given $\gamma>1$ and $\varepsilon>0$, denote 
	\begin{align*}
		R_{\gamma}=\Big\{(y_1,y_2,\cdots,y_{|\boldsymbol{m}|})\in\R^{|\boldsymbol{m}|d}:|y_{2i}|\ge\gamma|y_{2j-1}|,\textit{ for any }i,j=1,2,\cdots,\frac{|\boldsymbol{m}|}{2}\Big\},
	\end{align*}
	\begin{align*}
		T_{\varepsilon}=\Big\{(t_1,t_2,\cdots,t_{|\boldsymbol{m}|})\in D_{\boldsymbol{m}}:\frac{t_{2j}-t_{2j-1}}{t_{2i-1}-t_{2i-3}}\le \ln^{-\frac14}(1+\varepsilon^{-\frac{1}{2}}),\textit{ for any }i,j=1,2,\cdots,\frac{|\boldsymbol{m}|}{2}\Big\}
	\end{align*}
	and 
	\begin{align*}
		J_{\gamma,\varepsilon}=\bigg(\ell_{\varepsilon,H,d}^{(|\boldsymbol{k}|,N)}\bigg)^{|\boldsymbol{m}|}\frac{\boldsymbol{m}!}{(2\pi)^{|\boldsymbol{m}|d}}\int_{T_\varepsilon}\int_{R_\gamma} \prod^{|\boldsymbol{m}|/2}_{j=1}&\Big(\prod_{\ell=1}^d\big|y_{2j,\ell}\big|^{2k_\ell}\Big)\big|\widehat{f}\big(\varepsilon^{\frac{1}{2}}y_{2j}\big)-\widehat{f}(0)\big|^2\\&\times\exp\Big(\iota x\cdot y_1-\frac{1}{2}\Var\Big(\sum^{|\boldsymbol{m}|}_{j=1}y_j\cdot \big(X_{t_j}-X_{t_{j-1}}\big)\Big)\Big)dydt.
	\end{align*}
	Note that for $\varepsilon$ large enough in $T_\varepsilon$, we have 
	\begin{align*}
		t_{2j}-t_{2j-1}\le(t_{2i-1}-t_{2i-2}+t_{2i-2}-t_{2i-3})\ln^{-\frac14}(1+\varepsilon^{-\frac{1}{2}})\le3(t_{2i-1}-t_{2i-2})\ln^{-\frac14}(1+\varepsilon^{-\frac{1}{2}})
	\end{align*}
	for any $i,j=1,2,\cdots,\frac{|\boldsymbol{m}|}{2}$, which means that from Lemma \ref{5lma4}, in those $T_\varepsilon$, we can write
	\begin{align*}
		\Var\Big(\sum^{|\boldsymbol{m}|}_{j=1}y_j\cdot \big(X_{t_j}-X_{t_{j-1}}\big)\Big)=\big(1+\bar{\alpha}(\varepsilon)\big)\Big(\mathrm{Var}\Big(\sum_{j=1}^{|\boldsymbol{m}|/2}{y_{2j-1}}\cdot(X_{t_{2j-1}}-X_{t_{2j-3}})\Big)+\sigma\sum_{j=1}^{|\boldsymbol{m}|/2}|y_{2j}|^2(\Delta{t_{2j}})^{2H}\Big),
	\end{align*}
	where $\bar{\alpha}(\varepsilon)$ converges to $0$ as $\varepsilon\downarrow0$. Moreover, $T_\varepsilon$ and $R_\gamma$ can be expressed as 
	\begin{align*}
		R_{\gamma}=\Big\{(y_1,y_2,\cdots,y_{|\boldsymbol{m}|})\in\R^{|\boldsymbol{m}|d}:|y_{2i}|\ge\gamma\min_{1\le j\le |\boldsymbol{m}|/2}|y_{2j-1}|,\textit{ for any }i=1,2,\cdots,\frac{|\boldsymbol{m}|}{2}\Big\},
	\end{align*}
	\begin{align*}
		T_{\varepsilon}=\Big\{(t_1,t_2,\cdots,t_{|\boldsymbol{m}|})\in D_{\boldsymbol{m}}:t_{2j}-t_{2j-1}\le \frac{\max\limits_{1\le i\le\frac{|\boldsymbol{m}|}{2}}(t_{2i-1}-t_{2i-3})}{\ln^{\frac14}(1+\varepsilon^{-\frac{1}{2}})},\textit{ for any }i,j=1,2,\cdots,\frac{|\boldsymbol{m}|}{2}\Big\},
	\end{align*}
	using coordinate transform $s_{j}=t_{2j}-t_{2j-1}$ for any $j=1,2,\cdots,\frac{|\boldsymbol{m}|}{2}$, Lemma \ref{5lma1} and \ref{5lma2}, it is easy to get 
	\begin{align*}
		\lim\limits_{\varepsilon\downarrow0}J_{\gamma,\varepsilon}&=\Big(D_{H,d}\Big)^{\frac{|\boldsymbol{m}|}{2}}\bigg(\frac{\boldsymbol{m}!}{ 2^{\frac{|\boldsymbol{m}|}{2}}   (2\pi )^{\frac{|\boldsymbol{m}|d}{2} } } \bigg)\displaystyle \int_{D_{\frac{|\boldsymbol{m}|}{2}}}\displaystyle \int_{\mathbb{R}^{\frac{|\mathbf{m}|d}{2}}}  \exp\Big(\iota  x\cdot y_1-\frac{1}{2}\mathrm{Var}\Big(\sum_{j=1}^{|\boldsymbol{m}|/2}{y_{2j-1}}\cdot(X_{t_{2j-1}}-X_{t_{2j-3}})\Big)\Big)dydt	\\&=\Big(D_{H,d}\Big)^{\frac{|\boldsymbol{m}|}{2}}\bigg(\frac{\boldsymbol{m}!}{ 2^{\frac{|\boldsymbol{m}|}{2}}   (2\pi )^{\frac{|\boldsymbol{m}|d}{2} }   (\frac{\boldsymbol{m}}{2})!} \bigg)\displaystyle \int_{\prod\limits_{i=1}^n [a_i ,b_i]^{\frac{m_i}{2}} }\displaystyle \int_{\mathbb{R}^{\frac{|\mathbf{m}|d}{2}}}  \exp\Big(\iota  x\cdot \sum\limits^{|\boldsymbol{m}|/2}_{j=1}u_{j}-\frac{1}{2}\Var\big(\sum\limits^{|\boldsymbol{m}|/2}_{j=1}u_j\cdot X_{t_{2j-1}}\big)\Big)dudt\\&=\Big(D_{H,d}\Big)^{\frac{|\boldsymbol{m}|}{2}}\E \Big(\prod_{i=1}^n \big[W (L(b_i,x) ) - W (L (a_i,x))\big]^{m_i} \Big).
	\end{align*}
	So we only need to consider the upper bound of $\display\Big|\E \Big(\prod_{i=1}^n \big[F_{\varepsilon}(b_i) -F_{\varepsilon}(a_i)\big]^{m_i} \Big)-J_{\gamma,\varepsilon}\Big|$. We divide the procedure into three steps:
	
	{\bf Step I: } Denote 
	\begin{align*}
		J_{\gamma,\varepsilon,1}=\bigg(\ell_{\varepsilon,H,d}^{(|\boldsymbol{k}|,N)}\bigg)^{|\boldsymbol{m}|}\frac{\boldsymbol{m}!}{(2\pi)^{|\boldsymbol{m}|d}}\int_{T_\varepsilon}\int_{R_\gamma} \Big( \prod^{|\boldsymbol{m}|}_{j=1}\prod_{\ell=1}^d\big(\iota y_{j\ell}&-\iota y_{j+1,\ell}\big)^{k_\ell}\Big) \Big( \prod^{|\boldsymbol{m}|}_{j=1} (\widehat{f}\big(\varepsilon^{\frac{1}{2}}(y_j-y_{j+1})\big)-\widehat{f}(0))\Big)\\&\times\exp\Big(\iota x\cdot y_1-\frac{1}{2}\Var\Big(\sum^{|\boldsymbol{m}|}_{j=1}y_j\cdot \big(X_{t_j}-X_{t_{j-1}}\big)\Big)\Big)dydt.
	\end{align*}
	Using Lemma \ref{2lma2}, on $R_\gamma$ we have 
	\begin{align*}
		&\Big|\Big( \prod^{|\boldsymbol{m}|}_{j=1}\prod_{\ell=1}^d\big(\iota y_{j\ell}-\iota y_{j+1,\ell}\big)^{k_\ell}\Big) \Big( \prod^{|\boldsymbol{m}|}_{j=1} (\widehat{f}\big(\varepsilon^{\frac{1}{2}}(y_j-y_{j+1})\big)-\widehat{f}(0))\Big)-\prod^{|\boldsymbol{m}|/2}_{j=1}\Big(\prod_{\ell=1}^d\big|y_{2j,\ell}\big|^{2k_\ell}\Big)\big|\widehat{f}\big(\varepsilon^{\frac{1}{2}}y_{2j}\big)-\widehat{f}(0)\big|^2\Big|\\\le&C^{(4)}_{4,1}\sum_{j=1}^{|\boldsymbol{m}|/2}\Big(|y_{2j}|^{2|\boldsymbol{k}|}\big(\frac{|\varepsilon^{\frac{1}{2}}y_{2j}|^{2N}}{\gamma}\wedge1\big)+\frac{1}{\gamma}|y_{2j}|^{2|\boldsymbol{k}|}(|\varepsilon^{\frac{1}{2}}y_{2j}|^{2N}\wedge1)\Big)\times\prod_{i\ne j}|y_{2i}|^{|\boldsymbol{k}|}(|\varepsilon^{\frac{1}{2}}y_{2i}|\wedge1)^{2N},
	\end{align*}
	where using local non-determinism (\ref{local}), coordinate transform $s_j=t_{j}-t_{j-1},j=1,2,\cdots,|\boldsymbol{m}|$ with $t_0=0$ and Lemma \ref{2lma3}, we get 
	\begin{align}\label{delta1}
		\Big|J_{\gamma,\varepsilon,1}-J_{\gamma,\varepsilon}\Big|\le C^{(4)}_{4,2}\bigg(\ell_{\varepsilon,H,d}^{(|\boldsymbol{k}|,N)}\bigg)^{2}\sum_{j=1}^{|\boldsymbol{m}|/2}\int_{[0,T]^2}\int_{\R^{2d}}\Big(|y_{2j}|^{2|\boldsymbol{k}|}\big(\frac{|\varepsilon^{\frac{1}{2}}y_{2j}|^{2N}}{\gamma}\wedge1\big)+\frac{1}{\gamma}|y_{2j}|^{2|\boldsymbol{k}|}(|\varepsilon^{\frac{1}{2}}y_{2j}|^{2N}\wedge1)\Big)\nonumber\\\times\exp\Big(-\frac{\kappa_{H,|\boldsymbol{m}|}}{2}|y_{2j}|^2s^{2H}_{2j}\Big)dy_{2j}ds_{2j}\le C^{(4)}_{4,3}\Big(\frac{1}{\gamma^{(\frac{1}{2NH}-\frac{|\boldsymbol{k}|}{N}-\frac{d}{2N})\wedge1}}+\frac{1}{\gamma}\Big),
	\end{align}
	where $T=b_N$.
	
	{\bf Step II: } Denote 
	\begin{align*}
		J_{\gamma,\varepsilon,2}=\bigg(\ell_{\varepsilon,H,d}^{(|\boldsymbol{k}|,N)}\bigg)^{|\boldsymbol{m}|}\frac{\boldsymbol{m}!}{(2\pi)^{|\boldsymbol{m}|d}}\int_{D_{\boldsymbol{m}}}\int_{R_\gamma} \Big( \prod^{|\boldsymbol{m}|}_{j=1}\prod_{\ell=1}^d\big(\iota y_{j\ell}&-\iota y_{j+1,\ell}\big)^{k_\ell}\Big) \Big( \prod^{|\boldsymbol{m}|}_{j=1} (\widehat{f}\big(\varepsilon^{\frac{1}{2}}(y_j-y_{j+1})\big)-\widehat{f}(0))\Big)\\&\times\exp\Big(\iota x\cdot y_1-\frac{1}{2}\Var\Big(\sum^{|\boldsymbol{m}|}_{j=1}y_j\cdot \big(X_{t_j}-X_{t_{j-1}}\big)\Big)\Big)dydt.
	\end{align*}
	Note that for $t_0=0$, 
	\begin{align*}
		D_{\boldsymbol{m}}-T_\varepsilon\subseteq\bigcup_{i,j=1}^{\frac{|\boldsymbol{m}|}{2}}\Big\{(t_1,t_2,\cdots,t_{|\boldsymbol{m}|})\in D_{\boldsymbol{m}}:t_{2i}-t_{2i-1}\ge\ln^{-\frac{1}{2}}(1+\varepsilon^{-\frac{1}{2}})\textit{ or }t_{2j-1}-t_{2j-2}\le\ln^{-\frac{1}{4}}(1+\varepsilon^{-\frac{1}{2}})\Big\}
	\end{align*}
	and on $R_\gamma$,
	\begin{align*}
		\Big|\Big( \prod^{|\boldsymbol{m}|}_{j=1}\prod_{\ell=1}^d\big(\iota y_{j\ell}&-\iota y_{j+1,\ell}\big)^{k_\ell}\Big) \Big( \prod^{|\boldsymbol{m}|}_{j=1} (\widehat{f}\big(\varepsilon^{\frac{1}{2}}(y_j-y_{j+1})\big)-\widehat{f}(0))\Big)\Big|\le C^{(4)}_{4,4}\prod_{j=1}^{|\boldsymbol{m}|/2}|y_{2j}|^{2|\boldsymbol{k}|}\big(|\varepsilon^{\frac12}y_{2j}|\wedge1\big)^{2N}.
	\end{align*}
	Using local non-determinism (\ref{local}), coordinate transform $s_j=t_{j}-t_{j-1},j=1,2,\cdots,|\boldsymbol{m}|$ with $t_0=0$ and Lemma \ref{2lma3}, we get 
	\begin{align*}
		\Big|J_{\gamma,\varepsilon,1}-&J_{\gamma,\varepsilon,2}\Big|\le C^{(4)}_{4,5}\sum_{j=1}^{|\boldsymbol{m}|/2}\int_{0}^{\ln^{-\frac14}(1+\varepsilon^{-\frac12})}\int_{\R^d}\exp\Big(-\frac{\kappa_{H,|\boldsymbol{m}|}}{2}|y_{2j}|^2s^{2H}_{2j}\Big)dy_{2j}ds_{2j}\\+& C^{(4)}_{4,5}\bigg(\ell_{\varepsilon,H,d}^{(|\boldsymbol{k}|,N)}\bigg)^{2}\sum_{j=1}^{|\boldsymbol{m}|/2}\int_{\ln^{-\frac12}(1+\varepsilon^{-\frac12})}^{T}\int_{\R^d}|y_{2j}|^{2|\boldsymbol{k}|}\big(|\varepsilon^{\frac12}y_{2j}|\wedge1\big)^{2N}\exp\Big(-\frac{\kappa_{H,|\boldsymbol{m}|}}{2}|y_{2j}|^2s^{2H}_{2j}\Big)dy_{2j}ds_{2j},
	\end{align*}
	where using Lemma \ref{5lma2}, we get that 
	\begin{align}\label{delta2}
		\lim\limits_{\varepsilon\downarrow0}\Big|J_{\gamma,\varepsilon,1}-J_{\gamma,\varepsilon,2}\Big|=0.
	\end{align} 
	
	{\bf Step III: } 
	Like (\ref{division}), we can get that 
	\begin{align*}
		&\Big|\E \Big(\prod_{i=1}^n \big[F_{\varepsilon}(b_i) - F_{\varepsilon}(a_i)\big]^{m_i} \Big)-J_{\gamma,\varepsilon,2}\Big|\\\le&C^{(4)}_{4,6}\bigg(\ell_{\varepsilon,H,d}^{(|\boldsymbol{k}|,N)}\bigg)^{|\boldsymbol{m}|}\int_{D_{\boldsymbol{m}}}\int_{\R^{|\boldsymbol{m}|d}-R_\gamma}\prod_{j=1}^{|\boldsymbol{m}|/2}\Big(|y_{2j}|^{|\boldsymbol{k}|}\big(|\varepsilon^{\frac12}y_{2j}|\wedge1\big)^N\Big)^2\\&\qquad\qquad\qquad\quad\qquad\qquad\qquad\qquad\times\exp\Big(-\frac{\kappa_{H,|\boldsymbol{m}|}}{2}\sum_{j=1}^{|\boldsymbol{m}|}|y_j|^2\big({t_j}-{t_{j-1}}\big)^{2H}\Big)dydt\\&\qquad\qquad\qquad\qquad\qquad\qquad+C^{(4)}_{4,6}\begin{cases}
			\ln^{-\frac{1}{2}}(1+\varepsilon^{-\frac12}),&\textit{if }1-2NH=H(2|\boldsymbol{k}|+d);\\\varepsilon^{-\delta},&\textit{if }1-2NH<H(2|\boldsymbol{k}|+d)<1,
		\end{cases}
	\end{align*}
	Note that 
	\begin{align*}
		\R^{|\boldsymbol{m}|d}-R_\gamma\subseteq \bigcup_{i,j=1}^{|\boldsymbol{m}|/2}\Big\{(y_1,y_2,\cdots,y_{|\boldsymbol{m}|})\in\R^{|\boldsymbol{m}|d}:|y_{2j}|\le\gamma|y_{2i-1}|\Big\}.
	\end{align*}
	Using coordinate transform $s_j=t_j-t_{j-1},j=1,2,\cdots,|\boldsymbol{m}|$ with $t_0=0$, 
	\begin{align*}
		&\bigg(\ell_{\varepsilon,H,d}^{(|\boldsymbol{k}|,N)}\bigg)^{|\boldsymbol{m}|}\int_{D_{\boldsymbol{m}}}\int_{\R^{|\boldsymbol{m}|d}-R_\gamma}\prod_{j=1}^{|\boldsymbol{m}|/2}\Big(|y_{2j}|^{|\boldsymbol{k}|}\big(|\varepsilon^{\frac12}y_{2j}|\wedge1\big)^N\Big)^2\exp\Big(-\frac{\kappa_{H,|\boldsymbol{m}|}}{2}\sum_{j=1}^{|\boldsymbol{m}|}|y_j|^2\big(t_{j}-t_{j-1}\big)^{2H}\Big)dydt\\	\le&\gamma^{|\boldsymbol{k}|+N}\bigg(\ell_{\varepsilon,H,d}^{(|\boldsymbol{k}|,N)}\bigg)^{|\boldsymbol{m}|}\sum_{i_0,j_0=1}^{|\boldsymbol{m}|/2}\int_{[0,T]^{|\boldsymbol{m}|}}\int_{\R^{|\boldsymbol{m}|d}}\prod_{j=1\atop j\ne j_0}^{|\boldsymbol{m}|/2}\Big(|y_{2j}|^{|\boldsymbol{k}|}\big(|\varepsilon^{\frac12}y_{2j}|\wedge1\big)^N\Big)^2\Big(|y_{2j_0}|^{|\boldsymbol{k}|}\big(|\varepsilon^{\frac12}y_{2j_0}|\wedge1\big)^N\Big)\\&\qquad\qquad\qquad\qquad\qquad\qquad\qquad\qquad\times\Big(|y_{2i_0-1}|^{|\boldsymbol{k}|}\big(|\varepsilon^{\frac12}y_{2i_0-1}|\wedge1\big)^N\Big)\exp\Big(-\frac{\kappa_{H,|\boldsymbol{m}|}}{2}\sum_{j=1}^{|\boldsymbol{m}|}|y_j|^2s_j^{2H}\Big)dyds\\&\qquad\qquad\qquad\qquad\qquad\qquad\qquad\qquad\le C^{(4)}_{4,6,\gamma}\begin{cases}
			\ln^{-\frac{1}{2}}(1+\varepsilon^{-\frac12}),&\textit{if }1-2NH=H(2|\boldsymbol{k}|+d);\\\varepsilon^{-\delta},&\textit{if }1-2NH<H(2|\boldsymbol{k}|+d)<1,
		\end{cases}
	\end{align*}
	where the last inequality comes from Lemma \ref{5lma3}. So that given $\gamma$, we have 
	\begin{align}\label{delta3}
		\lim\limits_{\varepsilon\downarrow0}\Big|\E \Big(\prod_{i=1}^n \big[F_{\varepsilon}(b_i) - F_{\varepsilon}(a_i)\big]^{m_i} \Big)-J_{\gamma,\varepsilon,2}\Big|=0,
	\end{align}
	Finally, Proposition \ref{alleven} comes from (\ref{delta1}), (\ref{delta2}) and (\ref{delta3}).
\end{proof}
\section{Technical Lemmas}
 In this section we present the proofs of some technical lemmas which are used in the proof of Theorem \ref{thmdis}. Given $H>0$ and $d\ge1$, denote \begin{align*}
	\ell_{\varepsilon,H,d}^{(a,b)}=\begin{cases}
		\varepsilon^{-\frac{b}{2}}\ln^{-\frac12}(1+\varepsilon^{-\frac12}), &\textit{if } 1-2Hb=H(2a+d);\\
		\varepsilon^{\frac{1}{4}(2a+d-\frac{1}{H})}, &\textit{if } 1-2Hb<H(2a+d)<1.
	\end{cases}
\end{align*}
for any $a\ge0$, $b>0$ and $\varepsilon>0$. Then we have the following lemmas.
\begin{lemma}\label{5lma1}
	Assume $n\ge1$, $d\ge1$, $f\in\widetilde{\mS}_{d,n}$, $\sigma>0$, $M\ge0$, $\alpha(\varepsilon)$ a function satisfying $\displaystyle\lim\limits_{\varepsilon\downarrow0}\alpha(\varepsilon)=0$ and $k_\ell\ge0$, $\ell=1,2,\cdots,d$ with $\display|\boldsymbol{k}|=\sum_{\ell=1}^{d}k_\ell$. Then given $T>0$,
	\begin{align*}
		\lim\limits_{\varepsilon\downarrow0}&\bigg(	\ell_{\varepsilon,H,d}^{(|\boldsymbol{k}|,n)}\bigg)^2\int_{0}^{T}\int_{|x|\ge M}\big|\widehat{f}(\varepsilon^{\frac{1}{2}}x)-\widehat{f}(0)\big|^2\Big(\prod_{\ell=1}^{d}|x_\ell|^{2k_\ell}\Big)e^{-\frac{1}{2}(\sigma+\alpha(\varepsilon))|x|^2s^{2H}}dxds\\&=\begin{cases}
		\display\frac{1}{H\sigma^{\frac{1}{2H}}}\int_{\R^d}\Big|\int_{\R^d}f(z)\frac{1}{n!}\big(z\cdot x\big)^ndz\Big|^2\Big(\prod_{\ell=1}^{d}|x_\ell|^{2k_\ell}\Big)e^{-\frac{1}{2}|x|^2}dx,	&\text{if }H(2|\boldsymbol{k}|+d)=1-2nH;\\
		\display\frac{1}{\sigma^{\frac{1}{2H}}}\int_{0}^{\infty}\int_{\R^d}\big|\widehat{f}(x)-\widehat{f}(0)\big|^2\Big(\prod_{\ell=1}^{d}|x_\ell|^{2k_\ell}\Big)e^{-\frac{1}{2}|x|^2s^{2H}}dxds, &\text{if }1-2nH<H(2|\boldsymbol{k}|+d)<1.
		\end{cases}
	\end{align*}
\end{lemma}
\begin{proof}
Denote $\widetilde{T}^{\sigma}_\varepsilon=\big(\sigma+\alpha(\varepsilon)\big)^{\frac{1}{2H}}\varepsilon^{-\frac{1}{2H}}T$. We have 	
\begin{align*}
	&\int_{0}^{T}\int_{|x|\ge M}\big|\widehat{f}(\varepsilon^{\frac{1}{2}}x)-\widehat{f}(0)\big|^2\Big(\prod_{\ell=1}^{d}|x_\ell|^{2k_\ell}\Big)e^{-\frac{1}{2}(\sigma+\alpha(\varepsilon))|x|^2s^{2H}}dxds\\=&\varepsilon^{\frac{1}{2}(2k+d-\frac{1}{H})}\big(\sigma+\alpha(\varepsilon)\big)^{-\frac{1}{2H}}\int_{0}^{\widetilde{T}^{\sigma}_\varepsilon}\int_{|x|\ge M\varepsilon^{\frac12}}\big|\widehat{f}(x)-\widehat{f}(0)\big|^2\Big(\prod_{\ell=1}^{d}|x_\ell|^{2k_\ell}\Big)e^{-\frac{1}{2}|x|^2s^{2H}}dxds,
\end{align*}
where the result is easy to obtain when $1-2nH<H(2|\boldsymbol{k}|+d)<1$ because $\display\lim\limits_{\varepsilon\downarrow0}\widetilde{T}^{\sigma}_\varepsilon=+\infty$. When  $H(2|\boldsymbol{k}|+d)=1-2nH$, note that from Lemma \ref{2lma1}, for $\varepsilon$ small enough, 
\begin{align*}
	\varepsilon^{-n}\ln^{-1}(1&+\varepsilon^{-\frac12})\int_{0}^{T}\int_{|x|\le M}\big|\widehat{f}(\varepsilon^{\frac{1}{2}}x)-\widehat{f}(0)\big|^2\Big(\prod_{\ell=1}^{d}|x_\ell|^{2k_\ell}\Big)e^{-\frac{1}{2}(\sigma+\alpha(\varepsilon))|x|^2s^{2H}}dxds\\&\le C_{5,1}\ln^{-1}(1+\varepsilon^{-\frac12})\int_{0}^{T}\int_{|x|\le M}\big(|x|\wedge\varepsilon^{-\frac12}\big)^{2n}|x|^{2|\boldsymbol{k}|}e^{-\frac{1}{4}\sigma|x|^2s^{2H}}dxds\le C_{5,2}\ln^{-1}(1+\varepsilon^{-\frac12}).
\end{align*}
The results comes from  
\begin{align*}
	&\lim\limits_{\varepsilon\downarrow0}\varepsilon^{-n}\ln^{-1}(1+\varepsilon^{-\frac12})\int_{0}^{T}\int_{\R^d}\big|\widehat{f}(\varepsilon^{\frac{1}{2}}x)-\widehat{f}(0)\big|^2\Big(\prod_{\ell=1}^{d}|x_\ell|^{2k_\ell}\Big)e^{-\frac{1}{2}(\sigma+\alpha(\varepsilon))|x|^2s^{2H}}dxds\\=&\frac{1}{\sigma^{\frac{1}{2H}}}\lim\limits_{\varepsilon\downarrow0}\ln^{-1}(1+\varepsilon^{-\frac12})\int_{0}^{\widetilde{T}^{\sigma}_\varepsilon}\int_{\R^d}\big|\widehat{f}(x)-\widehat{f}(0)\big|^2\Big(\prod_{\ell=1}^{d}|x_\ell|^{2k_\ell}\Big)e^{-\frac{1}{2}|x|^2s^{2H}}dxds\\=&\frac{1}{H\sigma^{\frac{1}{2H}}}\lim\limits_{\varepsilon\downarrow0}\ln^{-1}(1+\widetilde{T}^{\sigma}_\varepsilon)\int_{0}^{\widetilde{T}^{\sigma}_\varepsilon}\int_{\R^d}\big|\widehat{f}(x)-\widehat{f}(0)\big|^2\Big(\prod_{\ell=1}^{d}|x_\ell|^{2k_\ell}\Big)e^{-\frac{1}{2}|x|^2s^{2H}}dxds\\=&\frac{1}{H\sigma^{\frac{1}{2H}}}\lim\limits_{\varepsilon\downarrow0}\ln^{-1}(1+\widetilde{T}^{\sigma}_\varepsilon)\int_{0}^{\widetilde{T}^{\sigma}_\varepsilon}\int_{\R^d}\big|\widehat{f}(x)-\widehat{f}(0)\big|^2\Big(\prod_{\ell=1}^{d}|x_\ell|^{2k_\ell}\Big)e^{-\frac{1}{2}|x|^2s^{2H}}dxds\\=&\frac{1}{H\sigma^{\frac{1}{2H}}}\int_{\R^d}\Big|\int_{\R^d}f(z)\frac{1}{n!}\big(z\cdot x\big)^ndz\Big|^2\Big(\prod_{\ell=1}^{d}|x_\ell|^{2k_\ell}\Big)e^{-\frac{1}{2}|x|^2}dx,
\end{align*}
where we use L'H\^opital's Rule and the fact $\display\lim\limits_{\varepsilon\downarrow0}\varepsilon^{-\frac{n}{2}}\big(\widehat{f}(\varepsilon^{\frac{1}{2}}u)-\widehat{f}(0)\big)=\int_{\R^d}f(v)\frac{(\iota u\cdot v)^n}{n!}dv$ for $f\in\mS_{d,n}$ in the last inequality.
\end{proof}
\begin{lemma}\label{5lma2}
	Suppose $n\ge1$, $d\ge1$, $\sigma>0$ and $k\ge0$. Let $0<t<T$. When $1-2nH\le H(2k+d)<1$,
	\begin{align*}
		\lim\limits_{\varepsilon\downarrow0}\bigg(	\ell_{\varepsilon,H,d}^{(k,n)}\bigg)^2\int_{t\ln^{-\alpha}(1+\varepsilon^{-\frac{1}{2}})}^{T}\int_{\R^d}\big(|\varepsilon^{\frac{1}{2}}x|\wedge1\big)^{2n}|x|^{2k}e^{-\frac{\sigma}{2}|x|^2s^{2H}}dxds=0
	\end{align*}
	for any $0<\alpha<1$. Moreover, for $f\in\widetilde{\mS}_{d,n}$, we have 
	\begin{align*}
		\lim\limits_{\varepsilon\downarrow0}\bigg(	\ell_{\varepsilon,H,d}^{(k,n)}\bigg)^2\int_{t\ln^{-\alpha}(1+\varepsilon^{-\frac{1}{2}})}^{T}\int_{\R^d}\big|\widehat{f}(\varepsilon^{\frac{1}{2}}x)-\widehat{f}(0)\big|^2\Big(\prod_{\ell=1}^{d}|x_\ell|^{2k_\ell}\Big)e^{-\frac{1}{2}(\sigma+\alpha(\varepsilon))|x|^2s^{2H}}dxdsdxds=0.
	\end{align*} 
\end{lemma}
\begin{proof}
	In the case $1-2nH=H(2k+d)$, the result comes from 
	\begin{align*}
		\varepsilon^{-n}&\ln^{-1}(1+\varepsilon^{-\frac12})\int_{t\ln^{-\alpha}(1+\varepsilon^{-\frac{1}{2}})}^{T}\int_{\R^d}\big(|\varepsilon^{\frac{1}{2}}x|\wedge1\big)^{2n}|x|^{2k}e^{-\frac{\sigma}{2}|x|^2s^{2H}}dxds\\&\le\ln^{-1}(1+\varepsilon^{-\frac12})\int_{t\ln^{-\alpha}(1+\varepsilon^{-\frac{1}{2}})}^{T}\int_{\R^d}|x|^{2k+2n}e^{-\frac{\sigma}{2}|x|^2t^{2H}\ln^{-2H\alpha}(1+\varepsilon^{-\frac{1}{2}})}dxds\le C_{5,3}t^{-1}\ln^{\alpha-1}(1+\varepsilon^{-\frac12}).
	\end{align*}
	And in the case $1-2nH< H(2k+d)<1$, the result comes from 
	\begin{align*}
		\varepsilon^{\frac{1}{2}(2k+d-\frac{1}{H})}\int_{t\ln^{-\alpha}(1+\varepsilon^{-\frac{1}{2}})}^{T}&\int_{\R^d}\big(|\varepsilon^{\frac{1}{2}}x|\wedge1\big)^{2n}|x|^{2k}e^{-\frac{\sigma}{2}|x|^2s^{2H}}dxds\\&\le\int_{t\varepsilon^{-\frac{1}{2H}}\ln^{-\alpha}(1+\varepsilon^{-\frac{1}{2}})}^{+\infty}\int_{\R^d}\big(|x|\wedge1\big)^{2n}|x|^{2k}e^{-\frac{\sigma}{2}|x|^2s^{2H}}dxds\to0
	\end{align*}
	as $\varepsilon\downarrow0$, where we use the fact $\display\lim\limits_{\varepsilon\downarrow0}\varepsilon^{-\frac{1}{2H}}\ln^{-\alpha}(1+\varepsilon^{-\frac{1}{2}})=+\infty$. The rest of the lemma comes from Lemma \ref{2lma1}. 
\end{proof}
\begin{lemma}\label{5lma3}
Suppose  $H>0$, $d\ge1$, $a\ge0$, $b>0$ and $1-2Hb\le H(2a+d)<1$. For any $m\ge1$, let $q_j\in\{0,1,2\}$, $j=1,2,\cdots,m$ satisfying $\display\sum_{j=1}^{m}q_j=m$. Denote $m_0,m_1,m_2$ the number of $0,1,2$ in $\{q_1,q_2,\cdots,q_m\}$.
Then  $m_0=m_2=\frac{m-m_1}{2}$ and given $T>0$, there exists a positive constant $C_{H,d,T}^{a,b,m}$, s.t. 
\begin{align}\label{case1}
	\Big(\ell_{\varepsilon,H,d}^{(a,b)}\Big)^m\prod_{j=1}^{m}\Big(\int_{0}^{T}\int_{\R^d}&\Big[|x_j|^a\Big(\big|\varepsilon^{\frac{1}{2}}x_j\big|^b\wedge1\Big)\Big]^{q_j}e^{-|x_j|^2t_j^{2H}}dx_jdt_j\Big)\nonumber\\&\le C_{H,d,T}^{a,b,m}\begin{cases}
		\ln^{-\frac{m_1}{2}}(1+\varepsilon^{-\frac12}),&\textit{if }1-2Hb=H(2a+d);\\\varepsilon^{-m_1\delta},&\textit{if }1-2Hb<H(2a+d)<1.
	\end{cases}
\end{align}
where $\display\delta=\frac18\Big\{(\frac{1}{H}-d)\wedge(2a+2b+d-\frac1H)\Big\}$. Moreover, there exists a positive constant $\widetilde{C}_{H,d,T}^{a,b,m}$, s.t. for any $0<t\le T$, 
\begin{align}\label{caseT}
	\Big(\ell_{\varepsilon,H,d}^{(a,b)}\Big)^m\prod_{j=1}^{m}\Big(\int_{0}^{t}\int_{\R^d}\Big[|x_j|^a\Big(\big|\varepsilon^{\frac{1}{2}}x_j\big|^b\wedge1\Big)\Big]^{q_j}e^{-|x_j|^2t_j^{2H}}dx_jdt_j\Big)\le \widetilde{C}_{H,d,T}^{a,b,m}t^{\frac{m}{2}(1-Hd)}.
\end{align}
\end{lemma}
\begin{proof}
	The equality $m_0=m_2=\frac{m-m_1}{2}$ is apparent. Using Lemma \ref{2lma3}, it is easy to get for all $0<t\le T$, 
	\begin{align}\label{m0}
		\int_{0}^{t}\int_{\R^d}&\Big[|x_j|^a\Big(\big|\varepsilon^{\frac{1}{2}}x_j\big|^b\wedge1\Big)\Big]^{0}e^{-|x_j|^2t_j^{2H}}dx_jdt_j=\int_{0}^{t}\int_{\R^d}e^{-|x_j|^2t_j^{2H}}dx_jdt_j\le C_{5,4}t^{1-Hd}
	\end{align}
	and 
		\begin{align}\label{m2}
		\int_{0}^{t}\int_{\R^d}&\Big[|x_j|^a\Big(\big|\varepsilon^{\frac{1}{2}}x_j\big|^b\wedge1\Big)\Big]^{2}e^{-|x_j|^2t_j^{2H}}dx_jdt_j\le C_{5,5}\begin{cases}
			\varepsilon^{-\frac{1}{2}(2a+d-\frac{1}{H})},&\textit{if }1-2Hb<H(2a+d)<1;\\
			\varepsilon^{b}\ln(1+\varepsilon^{-\frac{1}{2H}}t),&\textit{if }1-2Hb=H(2a+d).
		\end{cases}
	\end{align}
	Note that  
		\begin{align*}
		\int_{0}^{T}\int_{\R^d}\Big[|x_j|^a\Big(\big|\varepsilon^{\frac{1}{2}}x_j\big|^b\wedge1\Big)\Big]e^{-|x_j|^2t_j^{2H}}dx_jdt_j\le& C_{5,6}\begin{cases}
			\varepsilon^{\frac{b}{2}}T^{1-H(a+b+d)},&\textit{if }1-Hb>H(a+d);\\\varepsilon^{\frac{b}{2}}\ln(1+\varepsilon^{-\frac{1}{2H}}T),&\textit{if }1-Hb=H(a+d);\\\varepsilon^{-\frac{1}{2}(a+d-\frac{1}{H})},&\textit{if }1-Hb<H(a+d)<1;
		\end{cases}\\\le& C_{5,T}\begin{cases}
		\varepsilon^{\frac{b}{2}},&\textit{if }1-2Hb=H(2a+d);\\\varepsilon^{-\frac{1}{4}(2a+d-\frac{1}{H})-\delta},&\textit{if }1-2Hb<H(2a+d)<1.
		\end{cases}
	\end{align*}
	The inequality (\ref{case1}) is easy to get from $m_0=m_2=\frac{m-m_1}{2}$. Moreover, the inequality (\ref{caseT}) comes from (\ref{m0}), (\ref{m2}) and 
	\begin{align*}
		\int_{0}^{t}\int_{\R^d}&\Big[|x_j|^a\Big(\big|\varepsilon^{\frac{1}{2}}x_j\big|^b\wedge1\Big)\Big]e^{-|x_j|^2t_j^{2H}}dx_jdt_j\\&\le \varepsilon^{-\frac{1}{4}(2a+d-\frac{1}{H})}\int_{0}^{t}\int_{\R^d}|x_j|^{-\frac{d}{2}+\frac{1}{2H}}e^{-|x_j|^2t_j^{2H}}dx_jdt_j\le C_{5,7}\varepsilon^{-\frac{1}{4}(2a+d-\frac{1}{H})}t^{\frac{1}{2}(1-Hd)},
	\end{align*}
	where  we use decomposition $b=(-a-\frac{d}{2}+\frac{1}{2H})+(a+b+\frac{d}{2}-\frac{1}{H})$ in the above inequalities.
\end{proof}
Finally we give a lemma about $X\in G^{\sigma,H}_d$, which is used in the proof of Proposition \ref{alleven}.
\begin{lemma}\label{5lma4}
	Assume $X\in G^{\sigma,H}_d$ and $n\ge2$ is an even integer. There exists  a non-negative decreasing function $\varphi(\gamma):(1,\infty)\to\mathbb{R}$ with $\lim\limits_{\gamma\to\infty}\varphi(\gamma)=0$, s.t. 
	\begin{equation*}
		\begin{split}
			\bigg|\Var\Big(\sum_{j=1}^{n}x_j\cdot \big(X_{t_j}-X_{t_{j-1}}\big)\Big)&-\mathrm{Var}\Big(\sum_{j=1}^{n/2}{x_{2j-1}}\cdot(X_{t_{2j-1}}-X_{t_{2j-3}})\Big)-\sigma\sum_{j=1}^{n/2}|x_{2j}|^2(\Delta{t_{2j}})^{2H}\bigg|\\&\le\varphi(\gamma)\bigg(\mathrm{Var}\Big(\sum_{j=1}^{n/2}{x_{2j-1}}\cdot(X_{t_{2j-1}}-X_{t_{2j-3}})\Big)+\sigma\sum_{j=1}^{n/2}|x_{2j}|^2(\Delta{t_{2j}})^{2H}\bigg)
		\end{split}
	\end{equation*}
	for any $0=t_{-1}=t_0<t_1<t_2<\cdots<t_n<\infty$ with $\frac{t_{2i}-t_{2i-1}}{t_{2j-1}-t_{2j-2}}<\frac{1}{\gamma}$, $i,j=1,2,\dots,\frac{n}{2}$ and  $x_1,x_2,\cdots,x_n\in\R^d$, where $\Delta{t_{j}}=t_j-t_{j-1},j=1,2,\cdots,n$.
\end{lemma}
\begin{proof}
	Consider decomposition
	\begin{align*}
		\sum_{j=1}^{n}x_j\cdot \big(X_{t_j}-X_{t_{j-1}}\big)&=\sum_{j=1}^{n/2}{x_{2j-1}}\cdot(X_{t_{2j-1}}-X_{t_{2j-3}})+\sum_{j=1}^{n/2}x_{2j}\cdot(X_{t_{2j}}-X_{t_{2j-1}})\\&-\sum_{j=1}^{n/2}x_{2j-1}\cdot(X_{t_{2j-2}}-X_{t_{2j-3}}):=S_{odd,1}+S_{even,1}-S_{even,2}
	\end{align*}
	and
	\begin{align*}
		\Var\big(S_{odd,1}+S_{even,1}-S_{even,2}\big)&=\Var\big(S_{odd,1}\big)+\Var\big(S_{even,1}\big)-\Var\big(S_{even,2}\big)\\&+2\Cov\big(S_{odd,1}-S_{even,2},S_{even,1}\big)-2\Cov\big(S_{odd,1}-S_{even,2},S_{even,2}\big).
	\end{align*}
	Given any $\gamma>1$, using $\frac{t_{2i}-t_{2i-1}}{t_{2j-1}-t_{2j-2}}<\frac{1}{\gamma}$ for any $i,j=1,2,\cdots,n$, local non-determinism (\ref{local}) and Property (\ref{proc}), we can get 
	\begin{align*}
		\big|\Var\big(S_{even,2}\big)\big|\le \frac{C_{5,8}}{\gamma^{2H}}\sum_{j=1}^{n/2}|x_{2j-1}|^2\big(t_{2j-1}-t_{2j-3}\big)^{2H}\le\frac{C_{5,9}}{\gamma^{2H}}\mathrm{Var}\Big(\sum_{j=1}^{n/2}{x_{2j-1}}\cdot(X_{t_{2j-1}}-X_{t_{2j-3}})\Big),
	\end{align*}\begin{align*}
	\big|\Cov\big(S_{odd,1}-S_{even,2},S_{even,1}\big)\big|&\le C_{5,10}\beta(\gamma)\Big(\sum_{j=1}^{n/2}|x_{2j-1}|^2\big(t_{2j-1}-t_{2j-2}\big)^{2H}+\sum_{j=1}^{n/2}|x_{2j}|^2\big(t_{2j}-t_{2j-1}\big)^{2H}\Big)\\
	\le C_{5,11}&\beta(\gamma)\Big(\Var\Big(\sum_{j=1}^{n/2}{x_{2j-1}}\cdot(X_{t_{2j-1}}-X_{t_{2j-3}})\Big)+\sum_{j=1}^{n/2}|x_{2j}|^2\Var\Big(X_{t_{2j}}-X_{t_{2j-1}}\Big)\Big)
	\end{align*}
	and
	\begin{align*}
		\big|\Cov\big(S_{odd,1}-S_{even,2},S_{even,2}\big)\big|\le& C_{5,12}\beta(\gamma)\Big(\sum_{j=1}^{n/2}|x_{2j-1}|^2\big(t_{2j-1}-t_{2j-2}\big)^{2H}+\sum_{j=1}^{n/2}|x_{2j-1}|^2\big(t_{2j-2}-t_{2j-3}\big)^{2H}\Big)\\\le& C_{5,13}\beta(\gamma)\Var\Big(\sum_{j=1}^{n/2}{x_{2j-1}}\cdot(X_{t_{2j-1}}-X_{t_{2j-3}})\Big)
	\end{align*}
	for some $\beta(\gamma)$ satisfying $\lim\limits_{\gamma\to\infty}\beta(\gamma)=0$. Moreover, 
	\begin{align*}
		\Var\big(S_{even,1}\big)=\sum_{j=1}^{n/2}|x_{2j}|^2\Var\Big(X_{t_{2j}}-X_{t_{2j-1}}\Big)+2\sum_{i<j}\Cov\Big(x_i\cdot\big(X_{t_{2i}}-X_{t_{2i-1}}\big),x_j\cdot\big(X_{t_{2j}}-X_{t_{2j-1}}\big)\Big),
	\end{align*}
	where using Property (\ref{proc}) we get 
	\begin{align*}
		\Big|\sum_{i< j}\Cov\Big(x_i\cdot\big(X_{t_{2i}}-X_{t_{2i-1}}\big),x_j\cdot\big(X_{t_{2j}}-X_{t_{2j-1}}\big)\Big)\Big|\le& C_{5,14}\widetilde{\beta}(\gamma)\Big(\sum_{i< j}|x_i||x_j|\big(t_{2i}-t_{2i-1}\big)^H\big(t_{2j}-t_{2j-1}\big)^H\Big)\\\le& C_{5,15}\widetilde{\beta}(\gamma)\sum_{j=1}^{n/2}|x_{2j}|^2\Var\Big(X_{t_{2j}}-X_{t_{2j-1}}\Big)
	\end{align*}
	for some $\widetilde{\beta}(\gamma)$ satisfying $\lim\limits_{\gamma\to\infty}\widetilde{\beta}(\gamma)=0$ because of $\display\max\Big\{\frac{t_{2i}-t_{2i-1}}{t_{2j-1}-t_{2i}},\frac{t_{2j}-t_{2j-1}}{t_{2j-1}-t_{2i}}\Big\}<\frac{1}{\gamma}$ for any $1\le i<j\le\frac{n}{2}$. Finally, the desired result comes from 
	\begin{align*}
		\big(\sigma-\phi(\gamma)\big)\big(t_{2j}-t_{2j-1}\big)^{2H}\le\Var\Big(X_{t_{2j}}-X_{t_{2j-1}}\Big)\le\big(\sigma+\phi(\gamma)\big)\big(t_{2j}-t_{2j-1}\big)^{2H}
	\end{align*}
	 for some $\phi(\gamma)$ satisfying $\lim\limits_{\gamma\to\infty}\phi(\gamma)=0$ because of Property (\ref{prob}) and $\frac{t_{2j}-t_{2j-1}}{t_{2j-1}}<\frac{1}{\gamma},j=1,2,\cdots,\frac{n}{2}$.
\end{proof}
	
			$\begin{array}{cc}
			\begin{minipage}[t]{1\textwidth}
				{\bf Minhao Hong}\\
				School of Science, Shanghai Maritime University, Shanghai 201306, China\\
				\texttt{hongmhmath@163.com}
			\end{minipage}
			\hfill
		\end{array}$
\end{document}